%% file: midpoint.tex
\documentclass[12pt,reqno,a4paper]{amsart}
\usepackage[utf8]{inputenc}
\usepackage[T1]{fontenc}
\usepackage{fullpage}
\usepackage{hyperref} 
\usepackage{enumitem}
\usepackage{amsmath,amssymb,amsthm,color,mathscinet}
\usepackage{nameref}
\usepackage{graphicx}
\usepackage{subcaption}
\usepackage{moreverb}
\usepackage{siunitx}
\usepackage{pgfplots}
\pgfplotsset{compat=1.9}
\usepackage{pgfplotstable}

\input{macros}

\title{Convergence of an implicit-explicit midpoint scheme\\for computational micromagnetics}

\author{Dirk Praetorius, 
Michele Ruggeri, and 
Bernhard Stiftner}

\address{TU Wien, Institute for Analysis and Scientific Computing, Wiedner Hauptstr. 8-10/E101/4, 1040 Vienna, Austria}
\email{\{dirk.praetorius\,,\,michele.ruggeri\}@tuwien.ac.at}
\email{bernhard.stiftner@tuwien.ac.at \quad \rm (corresponding author)}

\date{\today}

\renewcommand{\subsection}[1]{\refstepcounter{subsection}\medskip{\bf\thesubsection.~#1.}}

\newtheorem{lemma}{Lemma}
\newtheorem{proposition}[lemma]{Proposition}
\newtheorem{algorithm}[lemma]{Algorithm}
\newtheorem{definition}[lemma]{Definition}
\newtheorem{remark}[lemma]{Remark}
\newtheorem{theorem}[lemma]{Theorem}

\begin{document}

\begin{abstract}
Based on lowest-order finite elements in space, we consider the numerical integration of the Landau--Lifschitz--Gilbert equation (LLG). The dynamics of LLG is driven by the so-called effective field which usually consists of the exchange field, the external field, and lower-order contributions such as the stray field. The latter requires the solution of an additional partial differential equation in full space. Following Bartels and Prohl (2006) (Convergence of an implicit finite element method for the Landau--Lifschitz--Gilbert equation.\ \textit{SIAM J.\ Numer.\ Anal.}\ 44(4):1405--1419), we employ the implicit midpoint rule to treat the exchange field. However, in order to treat the lower-order terms effectively, we combine the midpoint rule with an explicit Adams--Bashforth scheme. The resulting integrator is formally of second-order in time, and we prove unconditional convergence towards a weak solution of LLG. Numerical experiments underpin the theoretical findings.
\end{abstract}

\maketitle

\section{Introduction}

Time-dependent micromagnetic phenomena are usually modeled by the Landau--Lif\-schitz--Gilbert equation (LLG); see~\eqref{eq:LLG} below. This nonlinear partial differential equation (PDE) describes the behavior of the magnetization of some ferromagnetic body under the influence of the so-called effective field $\heff$. Global-in-time existence (and possible nonuniqueness) of weak solutions of LLG goes back to~\cite{visintin,Alouges1992}. For smooth problems, LLG admits a unique strong solution locally in time, provided the initial data are smooth (cf. \cite{cf2001}). Under similar restrictions the recent work \cite{DS2014} proves a strong-weak uniqueness principle for LLG. Unconditionally convergent numerical integrators have first been analyzed mathematically in~\cite{Bartels2006,Alouges2008}, where $\heff$ only consists of the exchange field (see Section~\ref{subsection:model} below). Here, \emph{unconditional convergence} means that convergence of the numerical integrator enforces no CFL-type coupling of the spatial mesh-size $h$ and the time-step size $k$. Moreover, convergence is understood in the sense that the sequence of discrete solutions for $h,k \rightarrow 0$ admits a subsequence which converges weakly in $\boldsymbol{H}^1$ towards a weak solution of LLG. The tangent plane integrator of~\cite{Alouges2008} requires to solve one \emph{linear} system per time-step (posed in the time-dependent discrete tangent plane), but is formally only first-order in time. Instead, the midpoint scheme of~\cite{Bartels2006} is formally second-order in time, but involves the solution of one nonlinear system per time-step.

Usually, the effective field $\heff$ which drives the dynamics of LLG couples LLG to other stationary or time-dependent PDEs; see, e.g.,~\cite{Carbou1998} for the coupling of LLG with the full Maxwell system,~\cite{Garcia-Cervera2007} for the electron spin diffusion in ferromagnetic multilayers, or~\cite{Carbou2011} for LLG with magnetostriction. In the case that the effective field involves stationary PDEs only (e.g., $\heff$ consists of exchange field, anisotropy field, applied exterior field, and self-induced stray field), the numerical analysis of the tangent plane integrator of~\cite{Alouges2008} has been generalized in~\cite{AKT2012,Bruckner2014}, where the lower-order contributions are treated explicitly in time by means of a forward Euler step.
It is proved that this preserves unconditional convergence. In~\cite{bpp} and~\cite{ellg,lppt}, the tangent plane integrator is adapted to the coupling of LLG with the full Maxwell system resp.\ the eddy current formulation. The works~\cite{bppr} and~\cite{Abert2014} extend the tangent plane integrator  to LLG with magnetostriction resp.\ LLG with spin diffusion interaction. Throughout,~\cite{bpp,ellg,lppt,bppr,Abert2014} prove unconditional convergence of the overall integrator. 
Moreover, one general theme of~\cite{bpp,lppt,bppr,Abert2014} is that the time marching scheme decouples the integration of LLG and the coupled PDE, so that ---despite the possibly nonlinear coupling~\cite{bppr,Abert2014}--- only two linear systems have to be solved per time-step. Moreover,~\cite{Abert2014} proves that the nodal projection step of the original tangent plane integrator~\cite{Alouges2008} can be omitted without losing unconditional convergence. For this projection-free variant of the tangent plane integrator, the recent work~\cite{FT17} also proves strong $\boldsymbol{H}^1$-convergence towards strong solutions.

As far as the midpoint scheme from~\cite{Bartels2006} is concerned, the work~\cite{Banas2008} provides an extended scheme for the Maxwell--LLG system. Even though the decoupling of the nonlinear LLG equation and the linear Maxwell system appears to be of interest for a time-marching scheme, the analysis of~\cite{Banas2008} treats only the full nonlinear system in each time-step.

The present work transfers ideas and results from~\cite{AKT2012,Bruckner2014} for the tangent plane integrator to the midpoint scheme. We prove that lower-order terms can be treated explicitly in time. This dramatically  lowers the computational work to solve the nonlinear system in each time-step of the midpoint scheme. Unlike~\cite{AKT2012,Bruckner2014}, however, the effective treatment of the lower-order terms requires an explicit two-step method (instead of the simple forward Euler method) to preserve the second-order convergence of the midpoint scheme. We prove that such an approach based on the Adams--Bashforth scheme guarantees unconditional convergence and remains formally of second-order in time. As an application of the proposed general framework, we discuss the discretization of the extended form of LLG \cite{zl2004,tnms2005} which is used to describe the current driven motion of domain walls.

\bigskip

\section{Model problem and discretization}
This section states the Gilbert formulation of LLG and extends the notion of a weak solution from~\cite{Alouges1992} to the present situation. Then, we introduce the notation for our finite element discretization and formulate the numerical integrator. Throughout, we employ standard Lebesgue and Sobolev spaces $L^2(\Omega)$ resp.\ $H^1(\Omega)$. For any Banach space $B$, we let $\boldsymbol{B}:= B^3$, e.g., $\LL{2}{\Omega} := (\L{2}{\Omega})^3$.

\subsection{Model problem}\label{subsection:model}
For a bounded Lipschitz domain $\Omega \subset \R^3$, initial data $\mm^0 \in \HH{1}{\Omega}$, final time $T>0$, and the Gilbert damping constant $\alpha>0$, the Gilbert form of LLG reads
\begin{subequations}\label{eq:LLG}
\begin{align}
\partial_t \mm &= - \mm \times \heff + \alpha\,\mm \times \partial_t \mm & &\textrm{ in } \Omega_T := \left(0,T\right) \times \Omega, \\
\partial_{\mathbf{n}} \mm &= \0 && \textrm{ on } \left(0,T\right) \times \partial \Omega, \\
\mm\l0r &= \mm^0 && \textrm{ in } \Omega.
\end{align}
\end{subequations}
With $\Cex>0$, $\ff: \R^3 \rightarrow \R^3$, and $\boldsymbol{\pi}: \HH{1}{\Omega} \cap \LL{\infty}{\Omega} \rightarrow \LL{2}{\Omega}$, the effective field reads 
\begin{equation} \label{eq:effective}
\heff := \Cex \Delta \mm \ + \ \ppi{\mm} + \ff;
\end{equation}
see Theorem~\ref{theorem:mainresult} for further assumptions on $\ppi{\cdot}$ and $\ff$. With the $\Ltwoshort$-scalar product $\prod{\vvarphi}{\ppsi} := \IntSet{\Omega} \vvarphi \cdot \ppsi\,\d{x}$ for all $\vvarphi,\ppsi \in \LL{2}{\Omega}$, consider the bulk energy
\begin{align}\label{eq:energyfunctional}
\EE{\mm}{\ff} := \frac{\Cex}{2} \norm{ \nabla \mm }{\LL{2}{\Omega}}^2 - \frac{1}{2} \prod{\ppi{\mm}}{\mm} - \prod{\ff}{\mm}.
\end{align}
With the convention $\prod{\mm \times \nabla  \mm}{\nabla \vvarphi} := \sum_{\ell=1}^3\prod{\mm \times \partial_{x_\ell} \mm}{\partial_{x_\ell} \vvarphi}$, we follow~\cite{Alouges1992} for the 
definition of a weak solution to~\eqref{eq:LLG}. Note that the variational formulation~\eqref{eq:variational} is just the weak formulation of~\eqref{eq:LLG} after integration by parts.

\begin{definition}\label{def:weak}
A function $\mm$ is a weak solution to~\eqref{eq:LLG} if the following properties \ref{item:weak1}--\ref{item:weak3} are satisfied: 
\begin{enumerate}[label=\rm{(\roman*)}]
\item \label{item:weak1} \label{cond:norm} $\mm \in \HH{1}{\Omega_T}$ and $\abs{\mm} = 1$ almost everywhere in $\Omega_T$;
\item \label{item:weak2} $\mm \lr{0} = \mm^0$ in the sense of traces;
\item \label{item:weak4} $\mm$ has bounded energy in the sense that there exists a constant $C>0$, which depends only on $\mm^0$, $\ppi{\cdot}$, and $\ff$, such that, for almost all $\tau \in (0,T)$, it holds that
\begin{align*}
\norm{\nabla \mm(\tau)}{\LL{2}{\Omega}}^2 + \Int{0}{\tau} \norm{\partial_t \mm}{\LL{2}{\Omega}}^2 \d{t} \leq C < \infty;
\end{align*}
\item \label{item:weak3} for all $\vvarphi \in \HH{1}{\Omega_T}$, it holds that
\begin{align}\label{eq:variational}
\begin{split}
\Int{0}{T} \prod{\partial_t \mm}{\vvarphi} \d{t} &= 
\Cex \Int{0}{T} \prod{\mm \times \nabla  \mm}{\nabla \vvarphi} \d{t} - \Int{0}{T} \prod{\mm  \times \ppi{\mm}}{\vvarphi} \d{t}\\
& \qquad -\Int{0}{T} \prod{\mm \times \ff}{\vvarphi} \d{t}   + 
\alpha \Int{0}{T} \prod{\mm \times \partial_t \mm}{\vvarphi} \d{t}.
\end{split}
\end{align}
\end{enumerate}
Moreover, $\mm$ is a physical weak solution if, additionally, it holds that
\begin{enumerate}[label=\rm{(\roman*)}]
\setcounter{enumi}{4}
\item \label{item:weak5} for almost all $\tau \in \left(0,T\right)$, it holds that
\begin{align}\label{eq:energyestimate}
\mathcal{E} \left(\mm \lr{\tau},\ff \lr{\tau} \right) + 
\alpha \Int{0}{\tau} \left\Vert \partial_t \mm \right\Vert_{\LL{2}{\Omega}}^2 \d{t} + 
\Int{0}{\tau} \prod{\partial_t \ff}{\mm} \d{t} \leq \mathcal{E} \left(\mm^0 ,\ff \lr{0} \right).
\end{align}
\end{enumerate}
\end{definition}

\subsection{Spatial discretization} \label{subsection:notation}
\def\Cmesh{C_{\rm mesh}}
Let $\Trian$ be a quasi-uniform triangulation of $\Omega$ into compact tetrahedra $K \in \Trian$, i.e., with the corresponding (global) mesh-size $h>0$, it holds that
\begin{align}
\Cmesh^{-1} h \leq \abs{ K }^{1/3} \leq \diam{K} \leq \Cmesh h \quad \textrm{for all } K \in \Trian.
\end{align}
Define the space of $\Trian$-piecewise affine, globally continuous functions
\begin{align*}
V_h := \left\{ v_h \in C \lr{\overline{\Omega}} : v_h|_{K} \in \mathcal{P}^1 \lr{K} \textrm{ for all } K \in \Trian \right\} \subset \H{1}{\Omega},
\end{align*}
and recall that $\VV_h := (V_h)^3$ denotes the corresponding space of vector fields. Let $\Nh = \left\{ \zz_1, \zz_2, \dots, \zz_N \right\}$ be the set of nodes of $\Trian$.
For $\zz_{\ell} \in \Nh$, let $\phi_{\ell} \in \Sh$ be the nodal basis function, i.e., $\phi_{\ell} \left(\zz_{\ell^{\prime}}\right) = \delta_{\ell,\ell^{\prime}}$ with Kronecker's delta. 
Let $\II_h:\boldsymbol{C}(\overline{\Omega}) \rightarrow \Vh$ be the nodal interpolation
\begin{align}\label{eq:nodalinterpolant}
 (\II_h\ppsi)(x) = \sum_{\ell=1}^N \phi_\ell(x)\ppsi(\zz_\ell)
 \quad\text{for all }\ppsi\in\boldsymbol{C}(\overline\Omega).
\end{align}%
Besides the standard $\Ltwoshort$-product $\prod{\cdot}{\cdot}$, define the approximate $\Ltwoshort$-product
\begin{align}\label{eq:approximateltwo}
\prodh{ \vvarphi }{ \ppsi } := 
\Sum{\ell = 1}{N} \beta_{\ell}\ \vvarphi \left(\zz_{\ell}\right) \cdot \ppsi\left(\zz_{\ell}\right) \textrm{ for all } \vvarphi,\ppsi \in \boldsymbol{C}(\overline{\Omega})
\,\text{ with }\, \beta_{\ell} := \IntSet{\Omega} \phi_{\ell} \d{x}.
\end{align}
With $\norm{\vvarphi_h}{h}^2 := \prodh{ \vvarphi_h }{ \vvarphi_h }$, elementary calculations and scaling arguments (see, e.g. \cite[Lemma~3.9]{Bartels2015}) show that
\begin{align}\label{eq:normhequiv1}
\norm{ \vvarphi_h}{\LL{2}{\Omega}} \leq 
\norm{ \vvarphi_h }{h} \leq \sqrt{5} \norm{ \vvarphi_h }{\LL{2}{\Omega}} \quad \textrm{for all } \vvarphi_h \in \Vh.
\end{align}
\def\Cprod{C_{\rm prod}}%
Moreover, there exists $\Cprod > 0$ which depends only on $|\Omega|$ and $\Cmesh$, such that
\begin{align}\label{eq:normhequiv2}
\abs{ \prodh{\vvarphi_h}{\ppsi_h} - \prod{\vvarphi_h}{\ppsi_h} } \leq 
\Cprod h^2 \norm{ \nabla \vvarphi_h }{\LL{2}{\Omega}} \norm{ \nabla \ppsi_h }{\LL{2}{\Omega}} \quad \textrm{for all } \vvarphi_h,\ppsi_h \in \Vh.
\end{align}
Define the discrete Laplacian $\Delta_h: \HH{1}{\Omega} \rightarrow \Vh$ via
\begin{align}\label{eq:deltah}
\prodh{\Deltah{ \vvarphi } }{\ppsi_h} = - \prod{\nabla \vvarphi}{\nabla \ppsi_h} \quad \textrm{for all } \ppsi_h \in \Vh
\end{align}
and similarly $\Ppi_h : \LL{2}{\Omega} \rightarrow \Vh$ via
\begin{align}\label{eq:Ppih}
\prodh{ \Ppih{ \vvarphi } }{\ppsi_h} = \prod{\vvarphi}{\ppsi_h} \quad \textrm{for all } \ppsi_h \in \Vh.
\end{align}
Note that $\prodh{ \Ppih{ \vvarphi } }{\ppsi} = \prodh{ \Ppih{ \vvarphi } }{\II_h\ppsi}
= \prod{ { \vvarphi } }{\II_h\ppsi}$ for all $\vvarphi,\ppsi \in \boldsymbol{C}(\overline{\Omega})$.

\subsection{Temporal discretization} 
Consider uniform time-steps $t_j := jk$ for $j=0,\dots,M$. Let $k:= T/M$ be the time-step size. 
For a Banach space $\boldsymbol B$ and a sequence $\left( \vvarphi^i \right)_{i=-1}^M$ in $\boldsymbol B$, define the mean-value and the discrete time-derivative by
\begin{align}\label{eq:discreteobjects}
\mid{\vvarphi}{i}:= \frac{\vvarphi^{i+1} + \vvarphi^i}{2} \quad \textrm{and} \quad \dt{\vvarphi}{i} := \frac{\vvarphi^{i+1} - \vvarphi^i}{k} \quad \textrm{for} \quad i = 0,\dots,M-1.
\end{align}
For $t \in [t_i,t_{i+1})$, define
\begin{subequations}\label{eq:discretefunctions}
\begin{align}
& \vvarphi_{k}^{=} \ltr := \vvarphi^{i-1}, \quad
\vvarphi_{k}^-\ltr := \vvarphi^i, \quad 
\vvarphi_{k}^+\ltr := \vvarphi^{i+1}, \quad 
\overline{\vvarphi}_{k}\ltr := \mid{\vvarphi}{i}, \textrm{ and } & \\
& \vvarphi_{k}\ltr := \vvarphi^{i+1} \frac{t-t_i}{t_{i+1}-t_i} + \vvarphi^i \frac{t_{i+1}-t}{t_{i+1}-t_i}. &
\end{align}
\end{subequations}
Note that $\vvarphi_{k}^{=}, \vvarphi_{k}^{-}, \vvarphi_{k}^{+}, \overline{\vvarphi}_{k} \in \L{2}{0,T;\boldsymbol{B}}$ and $\vvarphi_k \in \H{1}{0,T;\boldsymbol{B}}$ with $\partial_t \vvarphi_k (t) = \dt{\vvarphi}{i}$ for $t \in [t_i,t_{i+1})$.

\subsection{Implicit-explicit midpoint scheme}\label{subsection:extendedmidpoint}
Let $\boldsymbol{\Pi}_h : \Vh \times \Vh \times \Vh \rightarrow \LL{2}{\Omega}$ be an approximation to $\ppi{\cdot}$. For $i = 0,\dots,M-1$, let $\ff_h^{i+\frac12} \in \Vh$ be an approximation of $\ff(t_i+\frac{k}{2})$. 

The following algorithm has first been proposed and analyzed in~\cite{Bartels2006} for vanishing lower-order terms, i.e., $\heff = \Cex \Delta \mm$. This result caught a lot of attention in the literature~\cite{B06,Banas2008,Bavnas2008/09,Bavnas2012}, where lower-order terms are treated implicity in time. Since some of the lower-order terms (e.g., the stray field) are computationally expensive, our formulation of Algorithm~\ref{alg:midpoint} aims to treat these terms explicitly in time.


\begin{algorithm}\label{alg:midpoint}
Input: Approximation $\mm_h^{-1} :=\mm_h^0 \in\Vh$ of initial condition $\mm^0$.
\newline
Loop: For $0\le i\le M-1$, find $\mm_h^{i+1} \in \Vh$ such that, for all $\vvarphi_h \in \Vh$, it holds that
\begin{align}\label{eq:newmidpoint}
\prodh{\dth{\mm}{i}}{\vvarphi_h} =& -\Cex \prodh{ \midh{\mm}{i} \times \Deltah{ \midh{\mm}{i} }}{\vvarphi_h} - 
\prodh{\midh{\mm}{i} \times \Ppih{ \ppih{\mm_h^{i+1}}{\mm_h^i}{ \mm_h^{i-1}} }}{\vvarphi_h} 
\notag\\
&-\prodh{\midh{\mm}{i} \times \Ppih{\halfh{\ff}{i}} }{\vvarphi_h} + 
\alpha \prodh{\midh{\mm}{i} \times \dth{\mm}{i}}{\vvarphi_h}.
\end{align}
Output: Sequence $\mm_h^i$ of approximations to $\mm(t_i)$ for all $i = 0,1,\dots, M$. \qed
\end{algorithm}

We aim to choose $\boldsymbol{\Pi}_h$ such that the scheme is (formally) of second order and explicit in time for $i\ge1$.
If $\boldsymbol{\pi}_h(\ppsi_h) \approx \boldsymbol{\pi}(\ppsi_h)$, we mainly think of the following two choices:
\begin{subequations}\label{eq:different_approaches}
\begin{enumerate}[label=(\roman*)]
\item \label{item:apporoach2} the implicit midpoint rule
\begin{align}\label{eq:approach_mp}
\ppih{\mm_h^{i+1}}{\mm_h^i}{ \mm_h^{i-1}} := \boldsymbol{\pi}_h\Big(\frac{\mm_h^{i+1}+\mm_h^{i}}{2}\Big),
\end{align}
\item \label{item:apporoach1} the explicit Adams--Bashforth two-step method
\begin{align}\label{eq:appproach_ab}
\ppih{\mm_h^{i+1}}{\mm_h^i}{ \mm_h^{i-1}} := \frac32 \boldsymbol{\pi}_{h}(\mm_h^{i}) - \frac12 \boldsymbol{\pi}_{h}(\mm_h^{i-1}).
\end{align}
\end{enumerate}
The natural choice will be the implicit midpoint rule~\eqref{eq:approach_mp} for the first time-step $i=0$ and the explicit Adams--Bashforth method~\eqref{eq:appproach_ab} for all succeeding time-steps $i\ge1$.
We note that the Adams--Bashforth approach~\eqref{eq:appproach_ab} is computationally attractive, since the computation of $\mm_h^{i+1}$ by~\eqref{eq:newmidpoint} does only require the evaluation of $ \boldsymbol{\pi}_h(\mm_h^i)$, but does not involve $ \boldsymbol{\pi}_h(\mm_h^{i+1})$. Formally, however, it preserves the second-order accuracy of the overall integrator (cf. Section~\ref{section:experiment:academic}).
Moreover, the forward Euler scheme (which is employed and analyzed in~\cite{AKT2012,Bruckner2014} for the tangent plane integrator from~\cite{Alouges2008}) reads
\begin{align}\label{eq:approach_ee}
\ppih{\mm_h^{i+1}}{\mm_h^i}{ \mm_h^{i-1}} := \boldsymbol{\pi}_{h}(\mm_h^{i}),
\end{align}
but will generically lead to a reduced first-order convergence.
\end{subequations}

The next proposition transfers \cite[Lemma~3.1]{Bartels2006} from $\heff = \Cex \Delta \mm$ to the present situation. In particular, Algorithm~\ref{alg:midpoint} is well-defined. 

\begin{proposition}\label{prop:solvable}
Given $\mm_h^i \in \Vh$, the variational formulation~\eqref{eq:newmidpoint} admits a solution $\mm_h^{i+1} \in \Vh$. The latter satisfies $\abs{\mm_h^{i+1} \lr{\zz_{\ell}} } = \abs{\mm_h^{i} \lr{\zz_{\ell}} }$ for all $\zz_{\ell} \in \Nh$.
In particular, it holds that $\norm{\mm_h^i}{h} = \norm{\mm_h^0}{h}$ as well as $\norm{\mm_h^i}{\LL{\infty}{\Omega}} = \norm{\mm_h^0}{\LL{\infty}{\Omega}}$.
\end{proposition}

\begin{proof}
Let $\mm_h^i \in \Vh$. Let $\II_h$ be the nodal interpolant~\eqref{eq:nodalinterpolant}. Define $\FF_{hk}^i: \Vh \rightarrow \Vh$ by
\begin{align}\label{eq:defFFhki}
\FF_{hk}^{i} ( \vvarphi_h )  &:= \frac{2}{k} ( \vvarphi_h - \mm_h^i )
 + \II_h \left( \vvarphi_h \times \Ttetha(\vvarphi_h,\mm_h^i,\mm_h^{i-1}) \right) \textrm{ for all } \vvarphi_h \in \Vh ,
\end{align}
where
\begin{align*}
\begin{split}
\Ttetha(\vvarphi_h,\mm_h^i,\mm_h^{i-1}) &:= \Cex \Deltah{ \vvarphi_h} + \Ppih{ \boldsymbol{\Pi}_h ( 2 \vvarphi_h - \mm_h^i , \mm_h^i, \mm_h^{i-1} ) } \\
& \quad + \Ppih{ \halfh{\ff}{i} } - \frac{2\alpha}{k} ( \vvarphi_h - \mm_h^i ).
\end{split}
\end{align*}
Let $\ppsi_h \in \Vh$ and suppose $\FF_{hk}^i ( \ppsi_h ) = \0$. Then, direct calculation shows that $\mm_h^{i+1} := 2\ppsi_h - \mm_h^i \in \Vh$ solves the variational formulation~\eqref{eq:newmidpoint}. With $\phi_{\ell}$ being the nodal basis function corresponding to some node $\zz_{\ell} \in \Nh$, it holds that
\begin{align*}
0 & = \prodh{\FF_{hk}^{i} ( \ppsi_h )}{\ppsi_h \lr{\zz_{\ell}} \phi_{\ell}} \stackrel{\eqref{eq:approximateltwo},\eqref{eq:defFFhki}}{ = } \frac{2 \beta_{\ell}}{k} ( \ppsi_h ( \zz_{\ell} ) - \mm_h^i ( \zz_{\ell} ) ) \cdot \ppsi_h ( \zz_{\ell} ) \\
& = \frac{\beta_{\ell}}{2k} ( \mm_h^{i+1} (\zz_{\ell}) - \mm_h^i ( \zz_{\ell} ) ) \cdot ( \mm_h^{i+1} ( \zz_{\ell} ) + \mm_h^i ( \zz_{\ell} ) ) 
= \frac{\beta_{\ell}}{2k} ( | \mm_h^{i+1} ( \zz_{\ell} ) |^2 - | \mm_h^{i} ( \zz_{\ell} ) |^2 ).
\end{align*}
This proves $\abs{\mm_h^{i+1} ( \zz_{\ell} )} = \abs{\mm_h^{i} ( \zz_{\ell} )}$ for all $\zz_{\ell} \in \Nh$.
In particular, the definition~\eqref{eq:approximateltwo} of $\prodh{\cdot}{\cdot}$ yields $\norm{\mm_h^{i+1}}{h} = \norm{\mm_h^{i}}{h}$. 
Since the affine functions $\mm_h^i,\mm_h^{i+1}$ attain their $\boldsymbol{L}^\infty$ norm in one of the vertices $\zz_\ell\in\NN_h$, we also conclude $\norm{\mm_h^{i+1}}{\boldsymbol{L}^\infty(\Omega)} = \norm{\mm_h^{i}}{\boldsymbol{L}^\infty(\Omega)}$.

It remains to show that there exists 
$\ppsi_h \in \Vh$ with $\FF_{hk}^i \lr{\ppsi_h} = \0$. For $\vvarphi_h \in \Vh$, it holds that
\begin{align*}
\prodh{\FF_{hk}^{i} ( \vvarphi_h )}{\vvarphi_h} = \frac{2}{k} ( \norm{ \vvarphi_h }{h}^2 - \prodh{\mm_h^i}{\vvarphi_h} ) &\geq 
\frac{2}{k} \norm{ \vvarphi_h }{h} ( \norm{ \vvarphi_h }{h} - \norm{ \mm_h^i }{h} ).
\end{align*}
For $r^2 > \Sum{\ell = 0}{M} \beta_{\ell} \abs{\mm_h^i (\zz_{\ell})}^2 \stackrel{\eqref{eq:approximateltwo}}{=} \norm{ \mm_h^i }{h}^2$, it holds that
\begin{align*}
\prodh{\FF_{hk}^{i} ( \vvarphi_h )}{\vvarphi_h} \geq 0 \textrm{ for all } \vvarphi_h \in \Vh  \textrm{ with } \norm{ \vvarphi_h }{h} = r.
\end{align*}
Therefore, an application of the Brouwer fixed-point theorem (resp.\ its corollary \cite[Section 9.1, p.529]{Evans2010}) yields the existence of $\ppsi_h \in \Vh$ with $\norm{ \ppsi_h }{h} \leq r$ and $\FF_{hk}^i ( \ppsi_h ) = \0$. This concludes the proof. 
\end{proof}

\section{Convergence theorem}\label{section:mainresult}

\subsection{Statement and discussion of the main theorem}
The following theorem is the main result of the present work.

\begin{theorem}\label{theorem:mainresult}
{\rm (a)} Let $\mm^0 \in \HH{1}{\Omega}$ with $\abs{\mm^0} = 1$ a.e.\ in $\Omega$. Let $\mm_h^0 \in \Vh$ satisfy
\begin{subequations}\label{eq:condmzero}
\begin{align}
& \label{eq:mzeroweak} \mm_h^0 \rightharpoonup \mm^0 \quad \text{weakly in } \HH{1}{\Omega} \quad \textrm{as } h \rightarrow 0, \\
& \label{eq:mzerobounded} \norm{ \mm_h^0 }{\LL{\infty}{\Omega}} \leq C_0,
\end{align}
\end{subequations}
where $C_0$ does not depend on $h>0$. Let the approximation operator $\boldsymbol{\Pi}_h:\Vh^3 \rightarrow \LL{2}{\Omega}$ be stable in the sense that, for all $\vvarphi_h^1,\vvarphi_h^2,\vvarphi_h^3 \in \Vh$,
it holds that
\begin{align}\label{eq:pihbounded}
\norm{ \boldsymbol{\Pi}_h \left(\vvarphi_h^1,\vvarphi_h^2,\vvarphi_h^3\right) }{\LL{2}{\Omega}} \leq C_{\boldsymbol{\pi}} 
(1+\max_{i=1,2,3} \norm{ \vvarphi_h^i }{\LL{\infty}{\Omega}}) \Sum{i=1}{3} \norm{ \vvarphi_h^i }{\HH{1}{\Omega}},
\end{align}
where $C_{\boldsymbol{\pi}}$ does not depend on $h>0$. Define $\widehat{\ff}_{hk} \in \L{2}{0,T;\HH{1}{\Omega}}$ via $\widehat{\ff}_{hk} \left(t\right) := \halfh{\ff}{i}$ for $t \in [t_i,t_{i+1})$ and suppose that 
\begin{align}\label{eq:ffhkweak}
\widehat{\ff}_{hk} \rightharpoonup \ff \textrm{ weakly in } \LL{2}{\Omega_T} \ashkzero.
\end{align}
Then, \textashkzero \ unconditionally, there exists a subsequence of the postprocessed output $\mm_{hk}$  of Algorithm~\ref{alg:midpoint} which converges weakly in $\HH{1}{\Omega_T}$, towards some limit $\mm\in\HH{1}{\Omega_T}$, which additionally satisfies $\mm\in L^\infty(0,T;\HH{1}{\Omega})$ with $|\mm|=1$ \,a.e. in $\Omega_T$. Moreover, there exists a constant $C>0$ such that the same subsequence guarantees strong convergence $\mm_{hk}^\star\to\mm$ in $\LL{2}{\Omega_T}$ and uniform boundedness $\norm{\mm_{hk}^\star}{\LL\infty{\Omega_T}}\le C$ for all $\mm_{hk}^\star\in\{\mm_{hk}, \mm_{hk}^+,\mm_{hk}^-,\overline{\mm}_{hk}, \mm_{hk}^{=} \}$.

{\rm(b)} In addition, suppose that we can extract a further subsequence such that 
\begin{align}\label{eq:pihweak}
 \boldsymbol{\Pi}_h\big(\mm_{hk}^+,\mm_{hk}^-,\mm_{hk}^{=}\big)
 \rightharpoonup \ppi{\mm}
 \text{ weakly in }\LL{2}{\Omega_T} \text{ as $(h,k)\to(0,0)$}.
\end{align}
Then, the limit $\mm$ from~{\rm(a)} is a weak solution of LLG according to
Definition~\ref{def:weak}~\ref{item:weak1}--\ref{item:weak3}.

{\rm (c)} In addition, let $\ff \in \ConeL$ and $\boldsymbol{\pi}: \LL{2}{\Omega} \rightarrow \LL{2}{\Omega}$ be continuous, linear, and self-adjoint. Let~\eqref{eq:mzeroweak} and~\eqref{eq:ffhkweak}--\eqref{eq:pihweak} hold with strong convergence. Then, $\mm$ from~{\rm(a)} is even a physical weak solution in the sense of Definition~\ref{def:weak}~\ref{item:weak1}--\ref{item:weak5}.
\end{theorem}

\begin{remark}
If LLG admits a unique weak solution $\mm \in \HH{1}{\Omega_T}$ in the sense of Definition~\ref{def:weak}~\ref{item:weak1}--\ref{item:weak3}, then standard arguments prove indeed that all convergences of Theorem \ref{theorem:mainresult} hold for the full sequences $\mm_{hk}^\star\in\{\mm_{hk}, \mm_{hk}^+,\mm_{hk}^-,\overline{\mm}_{hk}, \mm_{hk}^{=} \}$ instead of only subsequences. For the exchange-only formulation of LLG with $\boldsymbol{h}_{\rm eff} = \Cex \Delta\mm$, nonuniqueness of solutions is shown in~\cite{Alouges1992}. However, the recent work~\cite{DS2014} proves a strong-weak uniqueness principle of the solutions of LLG, i.e., if a strong solution exists up to some time $T^{\ast}>0$, it is also the \emph{unique} weak solution up to time $T^{\ast}$.
\end{remark}

\begin{remark}
Unlike~\cite{Bartels2006}, where the fact that $\lvert\mm_h^0(\zz_{\ell})\rvert=1$ for all $\zz_{\ell} \in \Nh$ is used, our proof requires only $\mm_h^0\in\VV_h$.
However, this choice and hence $C_0 = 1$ is allowed and, for instance, met if $\mm_h^0$ is the nodal interpolant of $\mm^0 \in \HH{1}{\Omega} \cap \boldsymbol{C}(\overline{\Omega})$. Unlike~\cite[Appendix A]{Bruckner2014}, the present proof of the energy estimate of Definition~\ref{def:weak}~\ref{item:weak5}, does not require $\partial_t \ff = 0$ and allows weaker stability assumptions on $\boldsymbol{\pi}$. 
\end{remark}%

\begin{remark}
If we suppose the stronger estimate
\begin{align}\label{eq:pihbounded_alt}
\norm{ \boldsymbol{\Pi}_h \left(\vvarphi_{h}^{1},\vvarphi_{h}^{2},\vvarphi_{h}^{3}\right) }{\LL{2}{\Omega}} \leq C_{\boldsymbol{\pi}} 
\Sum{i=1}{3} \norm{ \vvarphi_{h}^{i} }{\HH{1}{\Omega}}
\textrm{ for all } \vvarphi_{h}^{1},\vvarphi_{h}^{2},\vvarphi_{h}^{3} \in \Vh
\end{align}
instead of~\eqref{eq:pihbounded}, then Theorem~\ref{theorem:mainresult} remains valid, even if~\eqref{eq:mzerobounded} fails to hold. To see this, note that~\eqref{eq:mzerobounded} is only used below to derive~\eqref{eq:beforegronwall1} from~\eqref{eq:mhiinfty}. Here, the stronger bound~\eqref{eq:pihbounded_alt} simplifies~\eqref{eq:mhiinfty} and guarantees~\eqref{eq:beforegronwall1} even if~\eqref{eq:mzerobounded} fails. Therefore, Lemma~\ref{lemma:dicreteenergy} and hence also Theorem~\ref{theorem:mainresult} remain valid.
\end{remark}

\begin{remark}\label{remark:consistency}
Suppose that  $\ppi{\cdot}$ satisfies stability in the sense of
\begin{align}\label{eq:pibounded}
 \norm{\ppi{\vvarphi}}{\LL{2}{\Omega}} \le C_\pi\big(1+\norm{\vvarphi}{\LL{\infty}{\Omega}}\big)\norm{\vvarphi}{\HH{1}{\Omega}}
 \quad\text{for all }\vvarphi\in \HH{1}{\Omega} \cap \LL{\infty}{\Omega}.
\end{align}
Suppose that $\boldsymbol{\Pi}_h$ satisfies stability~\eqref{eq:pihbounded} as well as the following consistency condition: Convergence $\vvarphi_{h}^1,\vvarphi_{h}^2,\vvarphi_{h}^3 \rightarrow \vvarphi$ in $\LL{2}{\Omega}$\ as $h \rightarrow 0$ implies
\begin{align}\label{eq:pihweak_alt}
\boldsymbol{\Pi}_h \left(\vvarphi_{h}^1,\vvarphi_{h}^2,\vvarphi_{h}^3\right) \rightharpoonup \ppi{\vvarphi} \textrm{ in } \LL{2}{\Omega}.
\end{align}
Then, the Lebesgue dominated convergence theorem implies~\eqref{eq:pihweak}.
To see this, recall $\mm_{hk}^\star\to\mm$ in $\LL{2}{\Omega_T}$ and $\norm{\mm_{hk}^{\star}}{L^{\infty}(0,T;\boldsymbol{H}^1(\Omega))}  \leq C < \infty$ as $(h,k)\to(0,0)$ for some appropriate subsequence for all $\mm_{hk}^\star\in\{\mm_{hk}, \mm_{hk}^+,\mm_{hk}^-,\overline{\mm}_{hk}, \mm_{hk}^{=} \}$. Since $\LL{2}{\Omega_T} = L^2(0,T;\LL{2}{\Omega})$, we can extract a further subsequence such that $\mm_{hk}^\star(t)\to\mm(t)$ for almost all $t\in[0,T]$. Let $\ppsi\in L^2(0,T;\LL{2}{\Omega})$. Then, assumption~\eqref{eq:pihweak_alt} implies
$$
\prod{\boldsymbol{\pi}{\big(\mm(t)\big)}-\boldsymbol{\Pi}_h\big({\mm_{hk}^+(t)},{\mm_{hk}^-(t)},{\mm_{hk}^=(t)}\big)}{\ppsi(t)}\to0
\quad\text{for almost all }t\in[0,T].
$$
Moreover, assumptions~\eqref{eq:pihbounded} and~\eqref{eq:pibounded} together with the properties of $\mm$ and $\mm_{hk}^\star$ imply
\begin{align*}
&\big|\prod{\boldsymbol{\pi}{\big(\mm(t)\big)}-\boldsymbol{\Pi}_h\big({\mm_{hk}^+(t)},{\mm_{hk}^-(t)},{\mm_{hk}^=(t)}\big)}{\ppsi(t)}\big|
\lesssim \norm{\ppsi(t)}{\LL{2}{\Omega}}
\end{align*}
Since $T<\infty$, the Lebesgue dominated converge theorem implies~\eqref{eq:pihweak}.
Moreover, strong convergence in~\eqref{eq:pihweak_alt} will also result in strong convergence in~\eqref{eq:pihweak}. 
We note that in many relevant situations (see Section~\ref{section:discussion} below), the assumptions~\eqref{eq:pibounded}--\eqref{eq:pihweak_alt} are usually easier to verify than~\eqref{eq:pihweak}.
\end{remark}%

The proof of Theorem~\ref{theorem:mainresult} extends that of \cite{Bartels2006} for $\heff = \Delta \mm$, but additionally employs certain ideas of \cite{Bruckner2014}. Roughly, it consists of the following four steps:
\begin{itemize}
\item 
Show that $\left(\mm_h^i\right)_{i=0}^M$ fulfills a discrete energy identity.
\item 
Successively extract subsequences of the postprocessed output~\eqref{eq:discretefunctions} of Algorithm \ref{alg:midpoint}, namely $\mm_{hk}^+$, $\mm_{hk}^-$, $\overline{\mm}_{hk}$, $\mm_{hk}^{=}$, and $\mm_{hk}$, which simultaneously converge (weakly) to the same limit $\mm\in\HH{1}{\Omega_T}$ \textashkzero.
\item 
Verify that the limit $\mm$ satisfies Definition~\ref{def:weak}~\ref{item:weak1}--\ref{item:weak3}.
\item 
Verify that the limit $\mm$ satisfies the energy estimate of Definition~\ref{def:weak}~\ref{item:weak5}.
\end{itemize}
To simplify our notation and for the rest of this section, we abbreviate 
\begin{align}\label{eq:defppih}
\ppihmi := \ppih{\mm_h^{i+1}}{\mm_h^i}{ \mm_h^{i-1}} \textrm{\quad \textrm{and} \quad} \ppihkm := \ppih{\mm_{hk}^{+}}{\mm_{hk}^{-}}{\mm_{hk}^{=}}.
\end{align}

\subsection{Discrete energy equality and weakly convergent subsequences}
The following discrete energy identity will prove the boundedness of the discrete solutions and will hence allow to extract weakly convergent subsequences.

\begin{lemma}\label{lemma:dicreteenergy}
\rm(a) For $i = 0, \dots, M-1$, it holds that
\begin{align}\label{eq:discreteenergy1}
\frac{\Cex}{2} \dtshort \norm{ \nabla \mm_h^{i+1} }{\LL{2}{\Omega}}^2 
+ \alpha \norm{ \dth{\mm}{i} }{h}^2 
= \prod{ \dth{\mm}{i} }{ { \ppihmi } + {\halfh{\ff}{i}}}.
\end{align}
\rm(b) For $i = 0, \dots, M-1$, it holds that
\begin{align}\label{eq:discreteenergy2}
\begin{split} 
&\frac{\Cex}{2} \norm{ \nabla \mm_h^{i+1} }{\LL{2}{\Omega}}^2 \!+\! 
\alpha k \Sum{j=0}{i} \norm{ \dth{\mm}{j} }{h}^2 
=\frac{\Cex}{2} \norm{ \nabla \mm_h^0 }{\LL{2}{\Omega}}^2
\!+\! k \Sum{j=0}{i} 
\prod{ \dth{\mm}{j}}{ { \ppihmj } \!+\! { \halfh{\ff}{j} }}.
\end{split}
\end{align}
\end{lemma}

\begin{proof}
Let $i \in \{ 0, \dots , M-1 \}$. With $\hh_h^{i+\frac12} := \Cex \Deltah{ \midh{\mm}{i} } + \Ppih{ \ppihmi } + \Ppih{ \halfh{\ff}{i}} \in \Vh$, equation~\eqref{eq:newmidpoint} reads as
\begin{align}\label{eq:newmidpointcompact}
\prodh{\dth{\mm}{i}}{\vvarphi_h} = - \prodh{\midh{\mm}{i} \!\times\! \hh_h^{i+\frac12} }{\vvarphi_h} + \alpha \prodh{\midh{\mm}{i} \!\times\! \dth{\mm}{i}}{\vvarphi_h} \textrm{ for all } \vvarphi_h \in \Vh.
\end{align}
Testing~\eqref{eq:newmidpointcompact} with $\vvarphi_h := \dth{\mm}{i} \in \Vh$, we obtain
\begin{align}\label{eq:specialtestenergy1}
\norm{ \dth{\mm}{i} }{h}^2 &\stackrel{\eqref{eq:newmidpointcompact}}{=} -\prodh{\midh{\mm}{i} \times \hh_h^{i+\frac12}}{\dth{\mm}{i}} = \prodh{\midh{\mm}{i} \times \dth{\mm}{i}}{\hh_h^{i+\frac12} }.
\end{align}
Testing~\eqref{eq:newmidpointcompact}, with $\vvarphi_h := \hh_h^{i+\frac12}  \in \Vh$, we obtain
\begin{align}\label{eq:specialtestenergy2}
&\prodh{\dth{\mm}{i}}{\hh_h^{i+\frac12} } \stackrel{\eqref{eq:newmidpointcompact}}{=} \alpha \prodh{\midh{\mm}{i} \times \dth{\mm}{i}}{\hh_h^{i+\frac12} }.
\end{align}
We multiply~\eqref{eq:specialtestenergy1} with $\alpha$ and get
\begin{align*}
&\alpha \norm{ \dth{\mm}{i} }{h}^2 
 \stackrel{\eqref{eq:specialtestenergy1}}{=} \alpha \prodh{\midh{\mm}{i} \times \dth{\mm}{i}}{\hh_h^{i+\frac12} }
\stackrel{\eqref{eq:specialtestenergy2}}{=} \prodh{\dth{\mm}{i}}{\hh_h^{i+\frac12} } \notag \\
&\qquad  = \Cex \prodh{\dth{\mm}{i}}{\Deltah{ \midh{\mm}{i}} } + 
\prodh{\dth{\mm}{i}}{\Ppih{ \ppihmi } + \Ppih{ \halfh{\ff}{i}} }.
\end{align*}
We calculate
\begin{align}\label{eq:energy3}
\begin{split}
 &\prodh{\dth{\mm}{i}}{\Deltah{ \midh{\mm}{i}} } 
 \stackrel{\eqref{eq:deltah}}{=} - \prod{\nabla \dth{\mm}{i}}{\nabla \midh{\mm}{i}} \\
 &\qquad
 \stackrel{\eqref{eq:discreteobjects}}{=}
-\frac{1}{2k} \prod{\nabla \mm_h^{i+1} - \nabla \mm_h^i}{\nabla \mm_h^{i+1} + \nabla \mm_h^i} 
= -\frac12 \dtshort \norm{ \nabla \mm_h^{i+1} }{\LL{2}{\Omega}}^2.
\end{split}
\end{align}
Combining the latter two identities and exploiting the definition~\eqref{eq:Ppih} of $\Ppih$, we prove~(a). To prove (b), we employ the telescopic series together with (a). 
This concludes the proof.
\end{proof}


\begin{lemma}\label{lemma:mhktinfty}
Let $\mm_h^0$, $\widehat{\ff}_{hk}$, $\boldsymbol{\Pi}_h$ resp.\ $\mm^0$, $\ff$, $\boldsymbol{\pi}$ satisfy the assumptions of Theorem~\ref{theorem:mainresult}~{\rm (a)}. Then, there exists $k_0>0$ such that for all $k<k_0$, the postprocessed output~\eqref{eq:discretefunctions} of Algorithm \ref{alg:midpoint} satisfies, for all $\mm^{\star}_{hk} \in \{\mm_{hk}, \mm_{hk}^+,\mm_{hk}^-,\overline{\mm}_{hk}, \mm_{hk}^{=} \}$,
\begin{align*}
\sup\limits_{\tau \in (0,T)} \norm{ \mm_{hk}^{\star} (\tau) }{\HH{1}{\Omega}}^2 + \sup\limits_{\tau \in (0,T)} \norm{ \partial_t \mm_{hk} }{\mathbf{L}^2 ( (0,\tau) \times \Omega )}^2 \leq C.
\end{align*}
The constant $C>0$ depends only on $\Cex$, $\alpha$, $\mm^0$, $\ff$, $\boldsymbol{\pi}$, $C_{\boldsymbol{\pi}}$, $T$, $\Omega$, and $\Cmesh$.
\end{lemma}

\begin{proof}
Let $\mm_{hk}^{\star} \in \{ \mm_{hk} , \mm_{hk}^+,\mm_{hk}^-,\overline{\mm}_{hk}, \mm_{hk}^{=} \}$ and $t \in \left[t_i, t_{i+1} \right)$. Proposition~\ref{prop:solvable} yields
\begin{eqnarray}\label{eq:mhkltwonorm}
\begin{split}
\norm{ \mm_{hk}^{\star} \lr{t} }{\LL{2}{\Omega}}^2 
& \lesssim 
\norm{ \mm_h^{i+1} }{\LL{2}{\Omega}}^2 + 
\norm{ \mm_h^{i} }{\LL{2}{\Omega}}^2 + 
\norm{ \mm_h^{i-1} }{\LL{2}{\Omega}}^2 \\
& \!\stackrel{\eqref{eq:normhequiv1}}{\simeq}
 \norm{ \mm_h^{i+1} }{h}^2 + 
\norm{ \mm_h^{i} }{h}^2 + 
\norm{ \mm_h^{i-1} }{h}^2 
\stackrel{\textrm{Prop. \ref{prop:solvable}}}{=}
3\,\norm{ \mm_h^{0} }{h}^2 \stackrel{\eqref{eq:normhequiv1}}{\simeq}
\norm{ \mm_h^{0} }{\LL{2}{\Omega}}^2.
\end{split}
\end{eqnarray}
By definition of $\mm_{hk}^{\star},$ it holds that
\begin{align}\label{eq:defsum1}
& \norm{ \nabla \mm_{hk}^{\star} (t) }{\LL{2}{\Omega}}^2 + 
\alpha \Int{0}{t} \norm{ \partial_t \mm_{hk} }{\LL{2}{\Omega}}^2 \d{t} \leq 
\norm{ \nabla \mm_{hk}^{\star} (t) }{\LL{2}{\Omega}}^2 + 
\alpha \Int{0}{t_{i+1}} \norm{ \partial_t \mm_{hk} }{\LL{2}{\Omega}}^2 \d{t} \notag \\
& \stackrel{\eqref{eq:discretefunctions}}{\lesssim} 
\norm{ \nabla \mm_h^{i+1} }{\LL{2}{\Omega}}^2 + \norm{ \nabla \mm_h^{i} }{\LL{2}{\Omega}}^2 + 
\norm{ \nabla \mm_h^{i-1} }{\LL{2}{\Omega}}^2 + \alpha k \Sum{j=0}{i} \norm{ \dth{\mm}{j} }{\LL{2}{\Omega}}^2 =: {\rm \rho}^{i}.
\end{align}
With the discrete energy equality~\eqref{eq:discreteenergy2} from Lemma~\ref{lemma:dicreteenergy}~
\rm(b), 
we estimate
\begin{eqnarray}\label{eq:mhkhone}
\begin{split}
{\rm \rho}^{i} & \stackrel{\eqref{eq:normhequiv1}, \eqref{eq:discreteenergy2}}{\lesssim} & 
\norm{ \nabla \mm_h^0 }{\LL{2}{\Omega}}^2 + 
k \Sum{j=0}{i} \abs{\prod{ \dth{\mm}{j}}{ \ppihmj + \halfh{\ff}{j}}}.
\end{split}
\end{eqnarray}
With the Young inequality and for arbitrary $\delta>0$, we get
\begin{eqnarray}\label{eq:mhiinfty}
{\rm \rho}^{i} &\lesssim&
\norm{ \nabla \mm_h^0 }{\LL{2}{\Omega}}^2 + \delta k \Sum{j=0}{i} \norm{ \dth{\mm}{j} }{\LL{2}{\Omega}}^2 + \frac{k}{\delta} \Sum{j=0}{i} \norm{ \ppihmj }{\LL{2}{\Omega}}^2 + \frac{1}{\delta} \Int{0}{t_{i+1}} \norm{ \widehat{\ff}_{hk} }{\LL{2}{\Omega}}^2 \d{t} \notag \\ 
& \stackrel{\eqref{eq:pihbounded}}{\lesssim} & \norm{ \nabla \mm_h^0 }{\LL{2}{\Omega}}^2 + \delta k \Sum{j=0}{i} \norm{ \dth{\mm}{j} }{\LL{2}{\Omega}}^2 \notag\\
&& \qquad +\frac{k}{\delta} (1+\max_{j=0,\dots,i+1} \norm{ \mm_h^j }{\LL{\infty}{\Omega}})^2 \Sum{j=0}{i+1} \norm{ \mm_h^j }{\HH{1}{\Omega}}^2 + 
\frac{1}{\delta} \Int{0}{t_{i+1}} \norm{ \widehat{\ff}_{hk} }{\LL{2}{\Omega}}^2 \d{t}.
\end{eqnarray}
Proposition \ref{prop:solvable} yields $\norm{ \mm_h^{i+1} }{\LL{\infty}{\Omega}} = \norm{ \mm_h^0 }{\LL{\infty}{\Omega}}$ for all $i \in \{0,\dots,M-1\}$. Combining~\eqref{eq:mhiinfty} with~\eqref{eq:mzerobounded}, we obtain
\begin{align}\label{eq:beforegronwall1}
\begin{split}
{\rm \rho}^{i} \stackrel{\eqref{eq:mzerobounded},\eqref{eq:mhkltwonorm}}{\lesssim} & \norm{ \nabla \mm_h^0 }{\LL{2}{\Omega}}^2 + \delta k \Sum{j=0}{i} \norm{ \dth{\mm}{j} }{\LL{2}{\Omega}}^2 + \frac{1}{\delta} \Int{0}{t_{i+1}} \norm{ \widehat{\ff}_{hk} }{\LL{2}{\Omega}}^2 \d{t} \\ 
& \quad + \frac{T}{\delta} \norm{\mm_h^0}{\LL{2}{\Omega}}^2
 + \frac{k}{\delta} \Sum{j=0}{i+1} \norm{\nabla \mm_h^j}{\LL{2}{\Omega}}^{2}.
 \end{split}
\end{align}
We choose $\delta \ll \alpha$ such that $\delta k \Sum{j=0}{i} \norm{ \dth{\mm}{j} }{\LL{2}{\Omega}}^2$ from~\eqref{eq:beforegronwall1} can be absorbed into the corresponding term of ${\rm \rho}^{i}$. Moreover, $k \norm{\nabla \mm_h^{i+1}}{\LL{2}{\Omega}}$ from~\eqref{eq:beforegronwall1} can be absorbed into ${\rm \rho}^{i}$ if $k<k_0$ is sufficiently small. Overall,~\eqref{eq:mhkhone}--\eqref{eq:beforegronwall1} result in
\begin{align}\label{eq:beforegronwall2}
\begin{split}
{\rm \rho}^{i} & \lesssim 
\norm{ \nabla \mm_h^0 }{\LL{2}{\Omega}}^2 + \frac{1}{\delta} \Int{0}{t_{i+1}} \norm{ \widehat{\ff}_{hk} }{\LL{2}{\Omega}}^2 \d{t} + \frac{T}{\delta} \norm{\mm_h^0}{\LL{2}{\Omega}}^2 + \frac{k}{\delta} \Sum{j=0}{i} \norm{\nabla \mm_h^j}{\LL{2}{\Omega}}^{2} \\
& \leq \norm{ \nabla \mm_h^0 }{\LL{2}{\Omega}}^2 + \frac{1}{\delta} \norm{ \widehat{\ff}_{hk} }{\LL{2}{\Omega_T}}^2 + \frac{T}{\delta} \norm{\mm_h^0}{\LL{2}{\Omega}}^2 + \frac{k}{\delta} \Sum{j=0}{i-1} {\rm \rho}^{j}
\end{split}
\end{align}
for $i=0,\dots,M-1$ and $k<k_0$. With the assumptions on $\mm_h^0$, $\widehat{\ff}_{hk}$, $\boldsymbol{\Pi}_h$ resp.\ $\mm^0$, $\ff$, $\boldsymbol{\pi}$, estimate~\eqref{eq:beforegronwall2} takes the form
\begin{align*}
{\rm \rho}^{i} \leq \alpha_0 + \beta \Sum{j=0}{i-1}  {\rm \rho}^{j} \quad \textrm{for all } i = 1,\dots,M-1,
\end{align*}
where $\abs{{\rm \rho}^{0}} < \alpha_0 < \infty$ and $0 <\beta \simeq k/\delta < \infty$. Thus, the discrete Gronwall lemma (e.g., \cite[Lemma~1.4.2]{Quarteroni1994}) yields that 
\begin{align*}
{\rm \rho}^{i} \leq \alpha_0 \exp\left(\Sum{j=0}{i-1} \beta\right) \lesssim \exp\left(\Sum{j=0}{i-1} \frac{k}{\delta}\right) \leq 
\exp(T/\delta).
\end{align*}
Together with~\eqref{eq:mhkltwonorm}--\eqref{eq:defsum1}, this concludes the proof.
\end{proof}

\begin{lemma}\label{lemma:extractsubsequences}
Let $\mm_h^0$ , $\widehat{\ff}_{hk}$, $\boldsymbol{\Pi}_h$ resp.\ $\mm^0$, $\ff$, $\boldsymbol{\pi}$ satisfy the assumptions of Theorem~\ref{theorem:mainresult} {\rm (a)}. Then, there exist $\mm \in \HH{1}{\Omega_T} \cap \L{\infty}{0,T;\HH{1}{\Omega}}$ as well as subsequences of the postprocessed output~\eqref{eq:discretefunctions} of Algorithm \ref{alg:midpoint} such that, for all $\mm_{hk}^{\star} \in  \{ \mm_{hk} , \mm_{hk}^+,\mm_{hk}^-,\overline{\mm}_{hk}, \mm_{hk}^{=} \}$,
\begin{enumerate}[label=\rm(\alph*)]
\item \label{item:weaksubsequence1} $\mm_{hk} \rightharpoonup \mm$ in $\HH{1}{\Omega_T}$,
\item \label{item:weaksubsequence2a} $\mm_{hk} \stackrel{\ast}{\rightharpoonup} \mm$ in $\L{\infty}{0,T;\HH{1}{\Omega}}$,
\item \label{item:weaksubsequence2} $\mm_{hk}^{\star} \rightharpoonup \mm$ in $\L{2}{0,T;\HH{1}{\Omega}}$,
\item \label{item:weaksubsequence3} $\mm_{hk}^{\star} \rightarrow \mm$ in $\LL{2}{\Omega_T}$,
\item \label{item:weaksubsequence4} $\mm_{hk}^{\star} \rightarrow \mm$ pointwise almost everywhere in $\Omega_T$,
\item \label{item:weaksubsequence5} $\mm_{hk}^{\star} \ltr \rightarrow \mm \ltr $ in $\LL{2}{\Omega}$ for $t \in \left[0,T\right)$ almost everywhere.
\end{enumerate}
where all convergences hold with respect to the same subsequence \textashkzero.
\end{lemma}

\begin{proof}
Let $\mm^{\star}_{hk} \in \{ \mm_{hk} , \mm_{hk}^+,\mm_{hk}^-,\overline{\mm}_{hk}, \mm_{hk}^{=} \}$. Lemma~\ref{lemma:mhktinfty} yields the existence of $C>0$ which is independent of $h,k>0$, such that $\norm{ \mm_{hk} }{\HH{1}{\Omega_T}} + \norm{ \mm_{hk}^{\star} }{\L{\infty}{\HH{1}{\Omega}}} \leq C<\infty$. With the Eberlein--\v{S}mulian theorem (resp.\ the Banach--Alaoglu theorem) and successive extraction of subsequences for all $\mm^{\star}_{hk} \in \{ \mm_{hk} , \mm_{hk}^+,\mm_{hk}^-,\overline{\mm}_{hk}, \mm_{hk}^{=} \}$, we get the convergences of \ref{item:weaksubsequence1}--\ref{item:weaksubsequence2} with possibly different limits. Let $\mm$ be the limit of~\ref{item:weaksubsequence1}. To prove~\ref{item:weaksubsequence3}, we use the Rellich--Kondrachov theorem and deduce from \ref{item:weaksubsequence1} that $\mm_{hk} \rightarrow \mm$ in $\LL{2}{\Omega_T}$. By definition of $\mm_{hk}^-$, we get
\begin{eqnarray*}
 \norm{ \mm_{hk} -\mm_{hk}^- }{\LL{2}{\Omega_T}}^2 
& \stackrel{\eqref{eq:discretefunctions}}{=} & \Sum{j=0}{M-1}  \Int{t_j}{t_{j+1}} \frac{(t - t_j)^2}{k^2} \norm{ \mm_{h}^{j+1} - \mm_{h}^j }{\LL{2}{\Omega}}^2 \d{t} \notag \\
& \stackrel{\eqref{eq:discreteobjects}}{=} & \Sum{j=0}{M-1} \Int{t_j}{t_{j+1}} (t - t_j)^2 \norm{ \partial_t \mm_{hk} \ltr }{\LL{2}{\Omega}}^2 \d{t}  \leq  k^2  \norm{ \partial_t \mm_{hk} }{\LL{2}{\Omega_T}}^2.
\end{eqnarray*}
Since $\partial_t \mm_{hk} \rightharpoonup \partial_t \mm$ and $\mm_{hk} \rightarrow \mm$ in $\LL{2}{\Omega_T}$, it follows that $\mm_{hk}^- \rightarrow \mm$ in $\LL{2}{\Omega_T}$. The convergences $\mm_{hk}^+,\mm_{hk}^-,\overline{\mm}_{hk} \rightarrow \mm$ in $\LL{2}{\Omega_T}$ follow analogously. Moreover,
\begin{align*}
\norm{ \mm_{hk}^- - \mm_{hk}^{=} }{\LL{2}{\Omega_T}} \leq \norm{ \mm_{hk}^+ - \mm_{hk}^{-} }{\LL{2}{\Omega_T}}  \longrightarrow 0
\end{align*}
concludes the proof of~\ref{item:weaksubsequence3}. Moreover, this concludes~\ref{item:weaksubsequence1}--\ref{item:weaksubsequence2}, since it identifies the limits. Upon successive extraction of further subsequences,~\ref{item:weaksubsequence4} and \ref{item:weaksubsequence5} are direct consequences of \ref{item:weaksubsequence3}.
\end{proof}

{\bf Convention.} {\em For the rest of this section, all limits are interpreted \textashkzero \ resp.\ $h \rightarrow 0$. Whenever limits of the postprocessed output~\eqref{eq:discretefunctions} of Algorithm \ref{alg:midpoint} are considered, this is understood in the sense of Lemma~
\ref{lemma:extractsubsequences}, i.e., all convergences hold with respect to the same subsequence.}
\subsection{Proof of Theorem~\ref{theorem:mainresult}~(a)}
In this subsection, we show that $\mm\in \HH{1}{\Omega_T}$ from Lemma~\ref{lemma:extractsubsequences} 
satisfies Definition~\ref{def:weak}~\ref{item:weak1}--\ref{item:weak4}. To that end, we adopt the notation of Lemma~\ref{lemma:extractsubsequences}. 

To see that $\abs{\mm} = 1$\ a.e. on $\Omega_T$, let $K \in \Trian$ and $t \in [t_i,t_{i+1})$. Then, for almost all $\xx \in K$, it holds that
\begin{align*}
\big| | \mm_{hk}^- (t, \xx ) | -  1 \big|  &= 
\big| | \mm_{h}^i (\xx) | -  | \mm^0 (\xx ) | \big| \\
&\leq \big| | \mm_{h}^i (\xx) | -  | \mm_h^0 (\xx ) | \big| + 
\big| | \mm_{h}^0 (\xx ) | -  | \mm^0 (\xx ) | \big|.
\end{align*}
Let $\zz_{\ell} \in \Nh$ be an arbitrary node of $K$. Since $\nabla \mm_h^i|_{K}$ is constant and $|\mm_h^i (\zz_{\ell})| = |\mm_h^0 (\zz_{\ell})|$ (cf. Proposition \ref{prop:solvable}), we get
\begin{align*}
\big| | \mm_h^i (\xx ) | -  | \mm_h^0 (\xx) | \big| & \leq \big| | \mm_h^i (\xx ) | -  | \mm_h^i (\zz_{\ell} ) | \big| 
+ 
\big| | \mm_h^0 (\zz_{\ell} ) | -  | \mm_h^0 ( \xx ) | \big| \\
& \leq \big| \nabla \mm_h^i|_{K} \big| | \xx - \zz_{\ell} | + 
\big| \nabla \mm_h^0|_{K} \big| | \xx - \zz_{\ell} |.
\end{align*}
Combining the last estimates, we obtain
\begin{align*}
\big| | \mm_{hk}^- (t, \xx ) | -  1 \big| \leq \big| \nabla \mm_{hk}^-|_{K} \big| | \xx - \zz_{\ell} | + 
\big| \nabla \mm_h^0|_{K} \big| | \xx - \zz_{\ell} | + \big| | \mm_h^0(\xx) | -  | \mm^0 (\xx) | \big|.
\end{align*}
Integrating this estimate over $\Omega_T$, we derive
\begin{align*}
& \norm{ \abs{ \mm } - 1 }{\L{2}{\Omega_T}} \leq
\norm{ \abs{\mm} - \abs{\mm_{hk}^-} }{\L{2}{\Omega_T}} + \norm{ \abs{\mm_{hk}^-} - 1}{\L{2}{\Omega_T}} \\
&  
\le \norm{ \mm - \mm_{hk}^-}{\LL{2}{\Omega_T}} + h \norm{ \nabla \mm_{hk}^{-} }{\LL{2}{\Omega_T}} + \sqrt{T} h \,\norm{ \nabla \mm_h^{0} }{\LL{2}{\Omega}} + \sqrt{T} \,\norm{ \mm_h^0 - \mm^0}{\LL{2}{\Omega}} .
\end{align*}
Since $\mm_{hk}^- \rightarrow \mm$ in $\LL{2}{\Omega_T}$, $\mm_{hk}^- \rightharpoonup \mm$ in $\L{2}{0,T;\HH{1}{\Omega}}$, and $\mm_h^0 \rightharpoonup \mm^0$ in $\HH{1}{\Omega}$, 
the right-hand side vanishes as $(h,k)\to(0,0)$. This concludes $ \norm{ \abs{ \mm } - 1 }{\L{2}{\Omega_T}} = 0$ and hence verifies Definition~\ref{def:weak}~\ref{item:weak1}.

To see that $\mm(0) = \mm^0$ in the sense of traces, note that  $\mm_{hk} \left( 0 \right) = \mm_h^0 \rightharpoonup \mm^0$ in $\HH{1}{\Omega}$. On the other hand, boundedness of the trace operator implies $\mm_{hk} \left( 0 \right) \rightharpoonup \mm \left( 0 \right)$ in $\HH{{1/2}}{\Omega}$. Since weak limits are unique and $\HH{1}{\Omega} \subset \HH{{1/2}}{\Omega}$, we conclude $\mm^0 = \mm \left(0\right)$ and verify Definition~\ref{def:weak}~\ref{item:weak2}. 

Bounded energy in the sense of Definition~\ref{def:weak}~\ref{item:weak4} is an immediate consequence of Lemma~\ref{lemma:mhktinfty}. This concludes the proof of Theorem~\ref{theorem:mainresult}~(a).\qed

\subsection{Proof of Theorem~\ref{theorem:mainresult}~(b)}
It only remains to verify Definition~\ref{def:weak}~\ref{item:weak3}. To this end, let $\vvarphi \in \boldsymbol{C}^{\infty} ( \overline{\Omega_T} )$. Let $\II_h$ be the nodal interpolant~\eqref{eq:nodalinterpolant}. Define $\vvarphi_h \in C^{\infty} ([0,T];\HH{1}{\Omega})$ by $\vvarphi_h (t) := \II_h( \vvarphi (t) )$. Then, $\vvarphi_h(t) \rightarrow \vvarphi(t)$ in $\boldsymbol{W}^{1,p}(\Omega)$ for all $p \in (3/2,\infty]$ and, consequently,
\begin{align}\label{eq:propinterpolant1}
\vvarphi_h \rightarrow \vvarphi \quad \textrm{in } \L{2}{0,T;\boldsymbol{W}^{1,p}(\Omega)} \quad
\textrm{for all } p \in (3/2,\infty].
\end{align}

{\bf Step~1.}\quad We collect some auxiliary results: First, $|\overline{\mm}_{hk} \times \vvarphi_h |_{\boldsymbol{H}^2 ( K )} \lesssim \norm{ \nabla \overline{\mm}_{hk} }{\LL{2}{K}} \linebreak \norm{ \nabla \vvarphi_h }{\LL{\infty}{K}}$ proves
\begin{align}\label{eq:interpolant7}
& \norm{ \nabla ( \overline{\mm}_{hk} \times \vvarphi_h ) - \nabla \II_h ( \overline{\mm}_{hk} \times \vvarphi_h ) }{\LL{2}{\Omega_T}} 
\lesssim h \Bigg( \Sum{j=0}{M-1} \SumSet{K \in \Trian} \Int{t_{j}}{t_{j+1}} | \overline{\mm}_{hk} \times \vvarphi_h |_{\boldsymbol{H}^2 ( K )}^{2} \d{t} \Bigg)^{1/2}\notag \\
&\quad\lesssim h \,\norm{ \nabla \overline{\mm}_{hk} }{\LL{2}{\Omega_T}}\norm{ \nabla \vvarphi_h }{\LL{\infty}{\Omega_T}}
\longrightarrow0.
\end{align}
The same argument proves
\begin{align}\label{eq*:interpolant7}
\norm{  \overline{\mm}_{hk} \times \vvarphi_h  - \II_h ( \overline{\mm}_{hk} \times \vvarphi_h ) }{\LL{2}{\Omega_T}} \longrightarrow0.
\end{align}
Second, it holds that
\begin{align}\label{eq:mphicalc}
\begin{split}
&\norm{ \mm \times \nabla \vvarphi - \overline{\mm}_{hk} \times \nabla \vvarphi_h }{\LL{2}{\Omega_T}} \\
& \quad \leq \norm{ \mm \times \left( \nabla \vvarphi - \nabla \vvarphi_h \right) }{\LL{2}{\Omega_T}} + \norm{ \left( \mm - \overline{\mm}_{hk} \right)\times\nabla \vvarphi_h  }{\LL{2}{\Omega_T}} \\
& \quad \lesssim \norm{ \mm }{\LL{\infty}{\Omega_{T}}}
\norm{ \nabla \vvarphi - \nabla \vvarphi_h }{\LL{2}{\Omega_T}} +   \norm{ \mm - \overline{\mm}_{hk} }{\LL{2}{{\Omega_T}}}\,\norm{ \nabla \vvarphi_h }{\LL{\infty}{\Omega_T}}
\longrightarrow0.
\end{split}
\end{align}
The same argument proves that
\begin{align}\label{eq*:mphicalc}
&\norm{ \mm \times \vvarphi - \overline{\mm}_{hk} \times \vvarphi_h }{\LL{2}{\Omega_T}} \longrightarrow0,
\end{align}
so that the combination with~\eqref{eq*:interpolant7} implies that
\begin{align}\label{eq**:mphicalc}
\norm{ \mm \times \vvarphi - \II_h(\overline{\mm}_{hk} \times \vvarphi_h)}{\LL{2}{\Omega_T}}
\longrightarrow0.
\end{align}

{\bf Step~2.}\quad
Plugging in the definitions of $\mm_{hk}$, $\mm_{hk}^+$, $\mm_{hk}^-$, $\overline{\mm}_{hk}$, $\mm_{hk}^{=}$, $\ppihkm$, $\widehat{\ff}_{hk}$, and $\vvarphi_h$ in~\eqref{eq:newmidpoint} and integrating in time, we obtain
\begin{eqnarray}\label{eq:variationalhk}
I_{hk}^1 &:=& \Int{0}{T} \prodh{\partial_t \mm_{hk}}{\vvarphi_{h}} \d{t} \notag \\
&\stackrel{\eqref{eq:newmidpoint}}{=} &  -\Cex \Int{0}{T} \prodh{\overline{\mm}_{hk} \times \Deltah{ \overline{\mm}_{hk} }} {\vvarphi_{h}} \d{t} - \Int{0}{T} \prodh{\overline{\mm}_{hk} \times \Ppih{ \ppihkm } }{\vvarphi_{h}} \d{t} \notag \\
& &-\Int{0}{T} \prodh{\overline{\mm}_{hk} \times \Ppih{\widehat{\ff}_{hk}} }{\vvarphi_{h}} \d{t} + \alpha \Int{0}{T} \prodh{\overline{\mm}_{hk} \times \partial_t \mm_{hk}}{\vvarphi_{h}} \d{t} \notag \\
& =: & - \Cex \ I_{hk}^2 - I_{hk}^3 - I_{hk}^4 + \alpha \ I_{hk}^5.
\end{eqnarray}
We aim to show that $I_{hk}^i$ for $i=1,\dots,5$ converge to their continuous counterparts in~\eqref{eq:variational}. 

{\bf Step~3 (Convergence of $\boldsymbol{I_{hk}^1}$, $\boldsymbol{I_{hk}^2}$, $\boldsymbol{I_{hk}^5}$).}\quad
With the auxiliary results from {\bf Step~1}, \cite[Section 3]{Bartels2006} proves that it holds
\begin{subequations}
\begin{align}
I_{hk}^1 &\longrightarrow 
\Int{0}{T} \prod{\partial_t \mm}{\vvarphi} \d{t} =: I^1, \\
I_{hk}^2 & \longrightarrow  
- \Int{0}{T} \prod{\mm \times \nabla \mm}{\nabla \vvarphi} \d{t} =:I^2,\\
I_{hk}^5 & \longrightarrow \Int{0}{T} \prod{\mm \times \partial_t \mm}{\vvarphi} \d{t} =:I^5.
\end{align}
\end{subequations}

{\bf Step~4 (Convergence of $\boldsymbol{I_{hk}^3}$ and $\boldsymbol{I_{hk}^4}$).}\quad
Using the definition~\eqref{eq:Ppih} of $\Ppi_h$ and $\boldsymbol{\Pi}_{hk}\rightharpoonup \ppi{\mm}$ from assumption~\eqref{eq:pihweak}, we obtain
\begin{align}\label{eq:calcint3}
\begin{split}
I_{hk}^3 
\stackrel{\eqref{eq:Ppih}}{=}  -\Int{0}{T} \prod{\ppihkm }{\II_h(\overline{\mm}_{hk} \times \vvarphi_{h})} \d{t} 
\stackrel{\eqref{eq**:mphicalc}}{\longrightarrow}
&-\Int{0}{T} \prod{\ppi{\mm}}{\mm \times \vvarphi} \d{t} 
\\&= 
\Int{0}{T} \prod{\mm \times \ppi{\mm}}{\vvarphi} \d{t} =: I^3.
\end{split}
\end{align}
By use of weak convergence $\widehat{\ff}_{hk}\rightharpoonup\ff$ from assumption~\eqref{eq:ffhkweak}, the same argument shows
\begin{align*}
I_{hk}^4 \stackrel{~\eqref{eq:Ppih}}{=} & -\Int{0}{T} \prod{\widehat{\ff}_{hk}}{\overline{\mm}_{hk} \times \vvarphi_{h}} \d{t}
 \stackrel{\eqref{eq**:mphicalc}}{\longrightarrow}- \Int{0}{T} \prod{\ff}{\mm \times \vvarphi} \d{t} = \Int{0}{T} \prod{\mm \times \ff}{\vvarphi} \d{t}.
\end{align*}

{\bf Step~5.}\quad Taking the limit $(h,k)\to(0,0)$ in~\eqref{eq:variationalhk}, we derive 
the variational formulation~\eqref{eq:variational}. By use of a density argument for $\boldsymbol{C}^\infty(\overline{\Omega_T})\subset\boldsymbol{H}^1(\Omega_T)$, this verifies Definition~\ref{def:weak}~\ref{item:weak3} and concludes the proof of Theorem~\ref{theorem:mainresult}~(b).
\qed

\subsection{Proof of Theorem~\ref{theorem:mainresult}~(c)}
Let $\tau \in (0,T)$. Let $0 \leq i \leq M-1$ such that $\tau \in [t_{i},t_{i+1})$. To simplify the notation, we set $\ff^i := \ff \left( t_i \right)$ with $i = 0,\dots,M$. Let $\ff_k$, $\ff_k^+$, $\ff_k^-$, and $\overline{\ff}_k$ the corresponding postprocessing~\eqref{eq:discretefunctions}. For $j \in \{0,1,\dots,M\}$, we obtain
\begin{eqnarray*}
&&\hspace*{-1cm} \EE{\mm_h^{j+1}}{\ff^{j+1}} - \EE{\mm_h^j}{\ff^j} \\
&\stackrel{\eqref{eq:energyfunctional}}{=} & \frac{\Cex k}{2} \dtshort \norm{\nabla\mm_h^{j+1}}{\LL{2}{\Omega}}^2 -
 \frac{1}{2} \prod{\ppi{\mm_h^{j+1}}}{\mm_h^{j+1}} + \frac{1}{2}\prod{\ppi{\mm_h^j}}{\mm_h^j} 
\\&&\qquad-\prod{\ff^{j+1}}{\mm_h^{j+1}} + \prod{\ff^j}{\mm_h^j} \\
& \stackrel{\eqref{eq:discreteenergy1}}{=} & -\alpha k \norm{\dth{\mm}{j}}{h}^2 
{- \frac{1}{2}\prod{\ppi{\mm_h^{j+1}}}{\mm_h^{j+1}} + \frac{1}{2}\prod{\ppi{\mm_h^j}}{\mm_h^j} + k \prodh{\dth{\mm}{j}}{\Ppih{ \ppihmj }}}\\
&& \qquad {- \prod{ \ff^{j+1}}{\mm_h^{j+1}} + \prod{ \ff^j}{\mm_h^j} + k \prodh{\dth{\mm}{j}}{\Ppih{ \halfh{\ff}{j} }}}
\\
&=:& -\alpha k \norm{\dth{\mm}{j}}{h}^2 + T_1 + T_2
\end{eqnarray*}
First, we consider $T_1$. Since $\boldsymbol{\pi}$ is linear and self-adjoint, simple calculations reveal that
\begin{align*}
&T_1 
\stackrel{\eqref{eq:Ppih}}{=} 
 - \frac{1}{2} \prod{\ppi{\mm_h^{j+1}}}{\mm_h^{j+1}} + \frac{1}{2}\prod{\ppi{\mm_h^j}}{\mm_h^j} + k \prod{\dth{\mm}{j}}{\ppihmj} 
\\&
= k \prod{\dth{\mm}{j}}{\ppihmj\!-\! \ppi{\midh{\mm}{j}}}
 \!-\! \frac{1}{2}\prod{\ppi{\mm_h^{j+1}}}{\mm_h^{j+1}} \!+\! \frac{1}{2}\prod{\ppi{\mm_h^j}}{\mm_h^j} + k \prod{\dth{\mm}{j}}{\ppi{\midh{\mm}{j}}}
\\&
= k \prod{\dth{\mm}{j}}{\ppihmj- \ppi{\midh{\mm}{j}}}.
\end{align*}
For $T_2$, we proceed similarly and obtain
\begin{eqnarray*}
T_2 
& \stackrel{\eqref{eq:Ppih}}{=} &
 -\prod{ \ff^{j+1}}{\mm_h^{j+1}} + \prod{ \ff^j}{\mm_h^j} + k \prod{\dth{\mm}{j}}{\halfh{\ff}{j}} \\
& = &k \prod{\dth{\mm}{j}}{\halfh{\ff}{j} - \mid{\ff}{j}}
 \underbrace{-\prod{ \ff^{j+1}}{\mm_h^{j+1}} + \prod{ \ff^j}{\mm_h^j} + k \prod{\dth{\mm}{j}}{\mid{\ff}{j}}}_{=:T_{21}}.
\end{eqnarray*}
Since
\begin{equation*}
\begin{split}
T_{21} & = -\prod{ \ff^{j+1}}{\mm_h^{j+1}} + \prod{ \ff^j}{\mm_h^j} + \frac{1}{2} \prod{\mm_h^{j+1}-\mm_h^j}{\ff^{j+1}+\ff^j} \\
& = - \frac{1}{2} \prod{\mm_h^{j+1}+\mm_h^j}{\ff^{j+1}-\ff^j}
= - k \prod{d_t\ff^{j+1}}{\midh{\mm}{j}},
\end{split}
\end{equation*}
we obtain
\begin{equation*}
\begin{split}
T_2 & = k \prod{\dth{\mm}{j}}{\halfh{\ff}{j} - \mid{\ff}{j}} - k \prodh{d_t \ff^{j+1}}{\midh{\mm}{j}}
\end{split}.
\end{equation*}
With these preliminary computations, we altogether obtain
\begin{equation*}
\begin{split}
&\EE{\mm_h^{j+1}}{\ff^{j+1}} - \EE{\mm_h^j}{\ff^j} + \alpha k \norm{ \dth{\mm}{j} }{h}^2 + k \prod{d_t\ff^{j+1}}{\midh{\mm}{j}} \\
& \qquad = k \prod{\dth{\mm}{j}}{\ppihmj - \ppi{\midh{\mm}{j}}} + k \prod{\dth{\mm}{j}}{\halfh{\ff}{j} - \mid{\ff}{j}}.
\end{split}
\end{equation*}
Summing over $0 \leq j \leq i$, the telescopic series proves
\begin{align*}
\begin{split}
& \EE{\mm_h^{i+1}}{\ff^{i+1}} + \alpha k \Sum{j=0}{i} \norm{ \dth{\mm}{j} }{h}^2 + k \Sum{j=0}{i} \prod{d_t\ff^{j+1}}{\midh{\mm}{j}} \\
&= \EE{\mm_h^0}{\ff^0} + k \Sum{j=0}{i} \prod{\dth{\mm}{j}}{\ppihmj - \ppi{\midh{\mm}{j}}}
+ k \Sum{j=0}{i} \prod{\dth{\mm}{j}}{\halfh{\ff}{j} - \mid{\ff}{j}}.
\end{split}
\end{align*}
This is equivalently written as
\begin{align}\label{eq:energyend1}
\begin{split}
& \EE{\mm_{hk}^{+} (\tau)}{\ff_k^{+}(\tau)} + \alpha \Int{0}{\tau} \norm{ \partial_t \mm_{hk} }{\LL{2}{\Omega}}^2 \d{t} + 
\Int{0}{t_{i+1}} \prod{\partial_t \ff_k}{\overline{\mm}_{hk}} \d{t}  \\
& \stackrel{\eqref{eq:normhequiv1}}{\leq} \EE{\mm_h^0}{\ff^0} + 
\Int{0}{t_{i+1}} \prod{\partial_t\mm_{hk}}{\ppihkm - \ppi{\overline{\mm}_{hk}}} \d{t} + \Int{0}{t_{i+1}} \prod{\partial_t \mm_{hk}}{\widehat{\ff}_{hk} - \overline{\ff}_k } \d{t}.
\end{split}
\end{align}
According to strong $\LL{2}{\Omega}$ convergence and no-concentration of Lebesgue functions (together with $\tau<t_{i+1}\le \tau+k$),
it holds that
\begin{align*}
 \Int{0}{t_{i+1}} \prod{\partial_t \ff_k}{\overline{\mm}_{hk}} \d{t}
 \longrightarrow \Int{0}{\tau} \prod{\partial_t \ff}{\mm} \d{t}.
\end{align*}
Together with strong convergence $\mm_h^0\to\mm^0$ of the initial data in $\HH{1}{\Omega}$, the same argument  implies
\begin{align*}
\EE{\mm_h^0}{\ff^0} + 
\Int{0}{t_{i+1}} \!\!\!\prod{\partial_t\mm_{hk}}{\ppihkm - \ppi{\overline{\mm}_{hk}}} \d{t} + \Int{0}{t_{i+1}} \!\!\!\prod{\partial_t \mm_{hk}}{\widehat{\ff}_{hk} - \overline{\ff}_k } \d{t}
\longrightarrow \EE{\mm^0}{\ff^0}.
\end{align*}
Weakly lower semicontinuity with respect to $\mm_{hk}^+$ proves, for all measurable $I\subseteq[0,T]$,
\begin{align*}
 &\int_{I}\Big(\mathcal E(\mm(\tau),\ff(\tau)) + \alpha \Int{0}{\tau} \norm{ \partial_t \mm }{\LL{2}{\Omega}}^2 \d{t}\Big)\d{\tau} 
 \\&\quad
 \le \liminf_{(h,k)\to(0,0)} \int_{I}\Big(\EE{\mm_{hk}^{+} (\tau)}{\ff_k^{+}(\tau)}\d{\tau}
 + \alpha \Int{0}{\tau} \norm{ \partial_t \mm_{hk} }{\LL{2}{\Omega}}^2 \d{t}\Big)\d{\tau}.
\end{align*}
Overall,~\eqref{eq:energyend1} thus leads to 
\begin{align*}
&\int_{I}\Big(\mathcal E(\mm(\tau),\ff(\tau)) + \alpha \Int{0}{\tau} \norm{ \partial_t \mm }{\Omega}^2 \d{t} + \Int{0}{\tau} \prod{\partial_t \ff}{\mm} \d{t}\Big)\d{\tau}
\le \int_{I}\EE{\mm^0}{\ff^0}\,\d{\tau}.
\end{align*}
Since $I\subseteq[0,T]$ was an arbitrary measurable subset, we obtain the estimate $\le$ pointwise almost everywhere in $[0,T]$ for the integrand. This verifies Definition~\ref{def:weak}~\ref{item:weak5} and concludes the proof of Theorem~\ref{theorem:mainresult}~(b).\qed

\def\aa{\mathbf{a}}
\section{Lower-order effective field contributions}\label{section:discussion}
In this section, we discuss some concrete examples for the general effective field contribution modeled by the operator $\boldsymbol{\pi}:\HH{1}{\Omega} \cap \LL{\infty}{\Omega} \rightarrow \LL{2}{\Omega}$.
We show that both the operators and their numerical approximations are covered by the abstract framework of Section~\ref{section:mainresult}.
We recall the notation introduced in Section~\ref{subsection:extendedmidpoint}:
For each $\boldsymbol{\pi}$ we denote by $\boldsymbol{\pi}_h$ the corresponding approximation, i.e., $\boldsymbol{\pi}_h(\vvarphi_h) \approx \ppi{\vvarphi_h}$ for any $\vvarphi_h\in\Vh$.
This is then used to define the generalized approximate operator $\boldsymbol{\Pi}_h$; see, e.g., \eqref{eq:different_approaches}.

\subsection{Classical contributions} \label{section:classicalcontributions}
The most common effective field contributions in micromagnetics are exchange, applied external field, magnetocrystalline anisotropy, and stray field, which already allow to describe a large variety of phenomena (cf. \cite{Hubert1998}).
The terms $\Cex \Delta \mm$ and $\ff$ of the abstract effective field~\eqref{eq:effective} clearly refer to the first two contributions.

In the case of uniaxial magnetocrystalline anisotropy, given the easy axis $\aa \in \R^3$ satisfying $|\aa| = 1$, we consider the operators
\begin{align*}
\boldsymbol{\pi}_h (\mm):= \boldsymbol{\pi} (\mm) := (\aa \cdot \mm) \aa \in \LL{2}{\Omega}
\end{align*}
defined for all $\mm \in \LL{2}{\Omega}$.
It is straightforward to show that, as far as the uniaxial anisotropy is concerned, all approaches~\eqref{eq:different_approaches} for the generalized approximation operator $\boldsymbol{\Pi}_h$ satisfy the assumptions of Theorem~\ref{theorem:mainresult}~{\rm (c)} and Remark~\ref{remark:consistency}.

General magnetocrystalline anisotropies $\boldsymbol{\pi}(\mm):= \nabla \phi(\mm)$ with $\phi \in C^1(\R^3)$ can be treated as in \cite[Section 4.1]{Bruckner2014} and satisfy the assumptions of Theorem \ref{theorem:mainresult}~{\rm (b)} and Remark \ref{remark:consistency}.

As for the stray field, it holds that $\ppi{\mm}= -\nabla u \vert_{\Omega}$, where the magnetostatic potential $u \in H^1(\R^3)$ solves the full space transmission problem
\begin{subequations} \label{eq:magnetostatic}
\begin{alignat}{2}
-\Delta u &= -\diver \mm &\quad& \textrm{in } \Omega, \\
-\Delta u &= 0 && \textrm{in } \R^3\setminus\overline{\Omega}, \\
u^{\mathrm{ext}} - u^{\mathrm{int}} &= 0 && \textrm{on } \partial \Omega, \\
(\nabla u^{\mathrm{ext}} - \nabla u^{\mathrm{int}})\cdot{\nn} &= -\mm\cdot\nn && \textrm{on } \partial \Omega, \\
u(\boldsymbol{x}) &= \mathcal{O}(\vert \boldsymbol{x} \vert^{-1}) && \textrm{as } \vert \boldsymbol{x} \vert \to \infty.
\end{alignat}
\end{subequations}
Here, the superscripts \emph{ext} and \emph{int} refer to the traces of $u$ on $\partial \Omega$ with respect to the exterior domain $\R^3 \setminus \overline{\Omega}$ and the interior domain $\Omega$, respectively, and $\nn$ is the outer normal vector on $\partial \Omega$ which points to $\R^3 \setminus \overline{\Omega}$. It can be shown that $\boldsymbol{\pi}: \LL{2}{\Omega} \to \LL{2}{\Omega}$ is a linear, bounded, and self-adjoint operator (cf. \cite[Proposition~3.1]{Pra04}).
However, $\boldsymbol{\pi}$ is nonlocal and behind the discrete operator $\boldsymbol{\pi}_h$ there is an effective discretization method for the transmission problem.
As an example, we consider the hybrid FEM-BEM approach from~\cite{Fredkin1990}; see also~\cite[Section 4.4.1]{Bruckner2014} for more details.
The starting point is the decomposition $\boldsymbol{\pi}(\mm) = -\nabla u \vert_{\Omega} = -\nabla u_1 - \nabla u_2$, where $u_1,u_2 \in \H{1}{\Omega}$ are the weak solutions of the boundary value problems
\begin{subequations}\label{eq:fredkin}
\begin{alignat}{2}
\Delta u_1 &= \diver \mm &\quad& \textrm{in } \Omega, \\
\partial_n u_1 &= \mm \cdot \nn && \textrm{on } \partial \Omega,
\end{alignat}
and
\begin{alignat}{2}
\Delta u_2 &= 0 &\quad& \textrm{in } \Omega, \\
\label{eq:fredkin4}
u_2 &= (K- 1/2) (u_1|_{\partial \Omega}) && \textrm{on } \partial \Omega,
\end{alignat}
\end{subequations}
respectively, where
\begin{align}\label{eq:doublelayer}
K \big( u_1|_{\partial \Omega} \big) (\xx) := 
\frac{1}{4\pi} \int_{\partial \Omega} \frac{(\xx-\yy)\cdot \nn(\yy)}{| \xx - \yy |^3} u_1^{\mathrm{int}}(\yy) \d{S(\yy)}
\end{align}
is the double-layer integral operator associated with the Laplace problem; see, e.g., \cite[Chapter 3]{ss2011}.

Taking the characterization~\eqref{eq:fredkin} into account, an effective approximation of $\boldsymbol{\pi}(\mm)$ can be obtained with the following algorithm.
\begin{algorithm} \label{alg:fk}
Input: Approximation $\mm_h \in \Vh$ of $\mm\in\LL{2}{\Omega}$.
\begin{enumerate}[label=\rm{(\roman*)}]
\item Compute $u_{1,h} \in V_h^{\star}:= \{ v_h \in V_h: \int_{\Omega} v_h \d{x} = 0 \}$ such that
\begin{align*}
\prod{\nabla u_{1,h}}{\nabla v_h} = \prod{\mm_h}{\nabla v_h} \quad \textrm{for all } v_h \in V_h^{\star}.
\end{align*}
\item Compute $g_h \in V_h^{\partial \Omega} := \{ v_h|_{\partial \Omega} : v_h \in V_h \}$ via the relation
\label{item:step2}
\begin{align*}
\prod{g_h}{v_h}_{\partial \Omega} = \prod{(K - 1/2)(u_{1,h}|_{\partial \Omega})}{v_h}_{\partial \Omega} \quad \textrm{for all } v_h \in V_h^{\partial \Omega}.
\end{align*}
\item Compute $u_{2,h} \in V_h$ such that $u_{2,h}|_{\partial \Omega} = g_h$ and
\begin{align*}
\prod{\nabla u_{2,h}}{\nabla v_h} = 0 \quad \textrm{for all } v_h \in V_h^0 := \{ v_h \in V_h: v_h|_{\partial \Omega} = 0 \}.
\end{align*}
\item Define $\boldsymbol{\pi}_h(\mm_h) := -\nabla u_{1,h} - \nabla u_{2,h} \in \LL{2}{\Omega}$.
\end{enumerate}
Output: Approximation $\boldsymbol{\pi}_h(\mm_h)$ of $\boldsymbol{\pi}(\mm)$.
\end{algorithm}
While~\cite[Section 4.4.1]{Bruckner2014} employed the Scott--Zhang projection from~\cite{Scott1990} in step~\ref{item:step2} of Algorithm~\ref{alg:fk}, we found in our numerical experiments that the $\LL{2}{\partial \Omega}$-orthogonal projection onto $V_h^{\partial \Omega}$ leads to better results on coarse meshes. Since, on quasi-uniform meshes, the $L^2$-orthogonal projection is $H^1$-stable and satisfies a first-order approximation property, the result of~\cite[Proposition~4.2]{Bruckner2014} remains valid; see also \cite[Section 4]{Goldenits2012}.
In particular, it follows that the approaches~\eqref{eq:different_approaches} for the generalized approximation operator $\boldsymbol{\Pi}_h$ fulfill the assumptions in Remark~\ref{remark:consistency}, even with strong convergence in~\eqref{eq:pihweak_alt}.
Overall, the stray field approximation of \cite{Fredkin1990} in the sense of Algorithm~\ref{alg:fk} thus fits in the setting of Theorem \ref{theorem:mainresult}~{\rm (c)}.

\subsection{Zhang--Li model for current-driven domain wall motion}\label{subsection:zl}
To take the transfer of the spin angular momentum between the local magnetization and spin-polarized currents into account, various extensions of the micromagnetic model have been considered.
In~\cite{tnms2005,zl2004}, the authors propose to add an additional torque term to LLG, which allows to model the current-driven motions of domain walls.
This extended LLG equation is usually referred to as Zhang--Li model.
Given the spin velocity vector $\vv \in \boldsymbol{C}(\overline{\Omega})$ and the constant $\xi>0$ (ratio of nonadiabaticity), the corresponding operator $\boldsymbol{\pi}:\HH{1}{\Omega} \cap \LL{\infty}{\Omega} \rightarrow \LL{2}{\Omega}$ (and its discretization) is defined by
\begin{align}\label{eq:defppizl}
\boldsymbol{\pi}_h(\mm) := \ppi{\mm} := \mm\times(\vv\cdot\nabla)\mm + \xi(\vv\cdot\nabla)\mm \quad \text{for all } \mm \in \HH{1}{\Omega} \cap \LL{\infty}{\Omega},
\end{align}
where $[(\vv\cdot\nabla)\mm]_j = \sum_{k=1}^3 v_k \partial_k m_j$ for all $j=1,2,3$.
In the mathematical literature, existence of (weak) solutions for an extended form of LLG with Zhang--Li spin transfer torque was studied in~\cite{Melcher2013}.
We now show that the operator~\eqref{eq:defppizl} satisfies the assumptions of Theorem~\ref{theorem:mainresult} {\rm (a)--(b)}. Note that part~\textrm{(c)} is clearly excluded, because $\boldsymbol{\pi}$ is nonlinear. It holds that
\begin{equation*}
\norm{\boldsymbol{\pi}_h(\mm_h)}{\LL{2}{\Omega}}
\leq \norm{\vv}{\LL{\infty}{\Omega}} \left(\xi + \norm{\mm_h}{\LL{\infty}{\Omega}}\right) \norm{\nabla\mm_h}{\LL{2}{\Omega}} \quad \textrm{for all }\mm_h \in \Vh.
\end{equation*}
This shows that, for any of the three approaches~\eqref{eq:different_approaches}, the generalized operator $\boldsymbol{\Pi}_h$ satisfies stability~\eqref{eq:pihbounded}. As for the consistency condition~\eqref{eq:pihweak}, let $\mm_{hk}^{\star} \in  \{ \mm_{hk} , \mm_{hk}^+,\mm_{hk}^-,\overline{\mm}_{hk}, \mm_{hk}^{=} \}$ be the postprocessed output~\eqref{eq:discretefunctions} of Algorithm \ref{alg:midpoint} and let $\mm\in\HH{1}{\Omega_T}$ be their common weak limit obtained from Theorem~\ref{theorem:mainresult}~\textrm{(a)}.
From the convergence properties of Lemma~\ref{lemma:extractsubsequences}, it follows that $(\vv \cdot \nabla) \mm^{\star}_{hk} \rightharpoonup (\vv \cdot \nabla) \mm$ in $\LL{2}{\Omega_T}$ and $\mm_{hk}^{\star} \times \vvarphi \rightarrow \mm \times \vvarphi$ in $\LL{2}{\Omega_T}$ for all $\vvarphi \in \LL{2}{\Omega_T}$. This implies
\begin{align*}
\mm_{hk}^{\star} \times (\vv \cdot \nabla) \mm_{hk}^{\star} \rightharpoonup \mm \times ( \vv \cdot \nabla ) \mm
\quad \textrm{and} \quad
(\vv \cdot \nabla) \mm^{\star}_{hk} \rightharpoonup (\vv \cdot \nabla) \mm
\quad \text{in } \LL{2}{\Omega_T}.
\end{align*}
This proves~\eqref{eq:pihweak}, so that the framework of Theorem~\ref{theorem:mainresult}~\textrm{(b)} applies.

\def\rr{\boldsymbol{r}}

\section{Iterative solution of nonlinear system}

Each time-step of Algorithm~\ref{alg:midpoint} requires the numerical solution of the nonlinear system~\eqref{eq:newmidpoint}. To that end, we follow \cite{Bartels2006} and employ the fixed-point iteration of the following algorithm. Up to a different stopping criterion, similar algorithms are also proposed in \cite{B06,Banas2008,Bavnas2008/09}.

\begin{algorithm}[Midpoint scheme with inexact solver]\label{alg:effectivemidpoint}
Input: Approximation $\mm_h^{-1} :=\mm_h^0 \in\Vh$ of initial condition $\mm^0$, $\hh_h^0 := \Cex \Delta_h \mm_h^0 + \Ppih{ \ppih{\mm_h^0}{\mm_h^0}{\mm_h^{-1}} } +  \Ppih{\ff_h^{\frac12}}$, tolerance $\epsilon>0$. 
\newline
Loop: For $0\le i\le M-1$, iterate the following steps {\rm (i)}--{\rm (iv)}:
\newline
{\rm (i)} Define $\eeta_h^{i,0} := \mm_h^i$, $\hh_h^{i,0} := \hh_h^i$.
\newline
{\rm (ii)} For $n=0,1,2,\dots$, repeat the following steps {\rm (ii-a)}--{\rm (ii-b)} until $\norm{\hh_h^{i,n+1} - \hh_h^{i,n}}{h} \leq \epsilon$:

{\rm (ii-a)} Find $\eeta_h^{i,n+1} \in \Vh$ such that, for all $\vvarphi_h \in \Vh$, it holds that
\begin{align}\label{eq:neweffectivemidpoint}
\frac{2}{k} \prodh{\eeta_h^{i,n+1}}{\vvarphi_h} + \prodh{ \eeta_h^{i,n+1} \times \hh_h^{i,n}}{\vvarphi_h} + 
\frac{\alpha}{2k} \prodh{\eeta_h^{i,n+1} \times \mm_h^i}{\vvarphi_h} = \frac{2}{k} \prodh{\mm_h^{i}}{\vvarphi_h}.
\end{align}

{\rm (ii-b)} Compute $\hh_h^{i,n+1} := \Cex \Delta_h \eeta_h^{i,n+1} +  \Ppih{ \ppih{2\eeta_h^{i,n+1} - \mm_h^i}{\mm_h^i}{ \mm_h^{i-1}} } + \Ppih{\ff_h^{i + \frac12}}$.
\newline
{\rm (iii)} \label{item:finalize} Define $\mm_h^{i+1} := 2 \eeta_h^{i,n+1} - \mm_h^i$.
\newline
{\rm (iv)} Compute $\hh_h^{i+1} := \Cex \Delta_h \mm_h^{i+1} + \Ppih{ \ppih{\mm_h^{i+1}}{\mm_h^{i+1}}{\mm_h^{i}} } +  \Ppih{\ff_h^{i+\frac12}}$.
\newline
Output: Sequence $\mm_h^i$ of approximations to $\mm(t_i)$ for all $i = 0,1,\dots, M$. \qed
\end{algorithm}

\begin{remark}\label{remark:propalgeffective}
We state some elementary properties of Algorithm \ref{alg:effectivemidpoint}. 

{\rm (i)} \label{item:propalgeffective1} The Lax--Milgram theorem yields that the linear system~\eqref{eq:neweffectivemidpoint} admits a unique solution. Let $\phi_{\ell}$ be the nodal basis function corresponding to some node $\zz_{\ell} \in \Nh$. Testing~\eqref{eq:neweffectivemidpoint} with $\vvarphi_h:= \eeta_h^{i,n+1}(\zz_{\ell}) \phi_{\ell} \in \Vh$, we obtain
\begin{align}
\frac{2\beta_{\ell}}{k} | \eeta_h^{i,n+1}(\zz_{\ell}) |^2 = 
\frac{2\beta_{\ell}}{k} \eeta_h^{i,n+1}(\zz_{\ell}) \cdot \mm_h^i (\zz_{\ell}) \leq 
\frac{2\beta_{\ell}}{k} | \eeta_h^{i,n+1}(\zz_{\ell}) | | \mm_h^i (\zz_{\ell}) |.
\end{align}
This proves $| \eeta_h^{i,n+1}(\zz_{\ell}) | \leq | \mm_h^i (\zz_{\ell}) |$ for all nodes $\zz_{\ell} \in \Nh$ and all $n \geq 0$. In particular, this yields $\norm{\eeta_h^{i,n+1}}{\infty} \leq \norm{\mm_h^i}{\infty}$ for all $n \geq 0$. 

{\rm (ii)}  \label{item:propalgeffective3} In contrast to Algorithm \ref{alg:midpoint}, the sequence $\left(\mm_h^i\right)_{i=0}^M$ from Algorithm \ref{alg:effectivemidpoint} satisfies 
\begin{align}\label{eq:newsecondmidpoint}
\prodh{\dth{\mm}{i}}{\vvarphi_h} &= 
-\Cex \prodh{ \midh{\mm}{i} \times \Deltah{ \midh{\mm}{i} }}{\vvarphi_h} - 
\prodh{\midh{\mm}{i} \times \Ppih{ \ppih{\mm_h^{i+1}}{\mm_h^i}{ \mm_h^{i-1}} }}{\vvarphi_h} 
\notag\\
&\quad -\prodh{\midh{\mm}{i} \times \Ppih{\halfh{\ff}{i}} }{\vvarphi_h} +
\alpha \prodh{\midh{\mm}{i} \times \dth{\mm}{i}}{\vvarphi_h}
 + \prodh{\midh{\mm}{i} \times \rr_h^{i,n} }{\vvarphi_h}.
\end{align}
for all $\vvarphi_h \in \Vh$, where $\rr_h^{i,n+1} := \hh_h^{i,n+1} - \hh_h^{i,n}$ and $n \geq 0$. As in Proposition \ref{prop:solvable}, the variational formulation~\eqref{eq:newsecondmidpoint} implies $\abs{\mm_h^{i+1} \lr{\zz_{\ell}} } = \abs{\mm_h^{i} \lr{\zz_{\ell}} }$ for all nodes $\zz_{\ell} \in \Nh$. 

{\rm (iii)}  \label{item:additionalassumptions} Under the assumption $k = \mathbf{o} (h^2)$, \cite[Lemma 4.1]{Bartels2006} proves that step {\rm (ii)} of Algorithm \ref{alg:effectivemidpoint} gives rise to a contraction $\Phi: \Vh \rightarrow \Vh$, where the contraction property holds with respect to $\norm{\cdot}{h}$.
In the presence of lower-order terms, the argument additionally requires the stability condition
\begin{align}\label{eq:additionalassumption}
\begin{split}
& \norm{\ppih{2\eeta_h^{i,n+1} - \mm_h^i}{\mm_h^i}{ \mm_h^{i-1}} - \ppih{2\eeta_h^{i,n} - \mm_h^i}{\mm_h^i}{ \mm_h^{i-1}} }{\LL{2}{\Omega}} \\
&\qquad \lesssim h^{-1} \norm{\eeta_h^{i,n+1} - \eeta_h^{i,n}}{\HH{1}{\Omega}},
\end{split}
\end{align}
which is satisfied for the approaches~\eqref{eq:different_approaches} if $\boldsymbol{\pi}_h: \HH{1}{\Omega} \rightarrow \LL{2}{\Omega}$ is Lipschitz continuous.
Then, the Banach fixed-point theorem applies and proves that $(\eeta_h^{i,n})_{n > 0}$ converges in $\LL{2}{\Omega}$ to a solution $\eeta_h^{i,\infty} \in \Vh$ and that $2\eeta_h^{i,\infty}-\mm_h^i \in \Vh$ solves~\eqref{eq:newmidpoint}.

{\rm (iv)} In consequence of {\rm (ii)}, it holds that
\begin{align}\label{eq:conviterates}
\norm{\eeta_h^{i,n+1} - \eeta_h^{i,n}}{\LL{2}{\Omega}} \longrightarrow 0.
\end{align}
In contrast to Algorithm \ref{alg:effectivemidpoint}, \cite[Algorithm 4.1]{Bartels2006} as well as \cite{Banas2008,Bavnas2008/09} use the stopping criterion $\norm{\eeta_h^{i,n+1} - \eeta_h^{i,n}}{\LL{2}{\Omega}} < \epsilon$. \cite[Algorithm A]{B06} uses the same stopping criterion as Algorithm~\ref{alg:effectivemidpoint}, however, only $\heff = \Cex \Delta \mm$ is considered. We note that our stopping criterion generically leads to less iterations. Together with~\eqref{eq:additionalassumption}, the inverse inequality yields that
\begin{align}\label{eq:solverstep}
\norm{\hh_h^{i,n+1} - \hh_h^{i,n}}{h} &\stackrel{\phantom{\eqref{eq:additionalassumption}}}{\lesssim}  
\norm{\Delta_h \eeta_h^{i,n+1} - \Delta_h \eeta_h^{i,n}}{h} \notag \\
& \quad + \norm{\ppih{2\eeta_h^{i,n+1} - \mm_h^i}{\mm_h^i}{ \mm_h^{i-1}} - \ppih{2\eeta_h^{i,n} - \mm_h^i}{\mm_h^i}{ \mm_h^{i-1}} }{\LL{2}{\Omega}} \notag \\
& \stackrel{\eqref{eq:additionalassumption}}{\lesssim} 
h^{-2} \norm{\eeta_h^{i,n+1} - \eeta_h^{i,n}}{\LL{2}{\Omega}} + 
h^{-1} \norm{\eeta_h^{i,n+1} - \eeta_h^{i,n}}{\HH{1}{\Omega}} \notag \\
& \stackrel{\phantom{\eqref{eq:additionalassumption}}}{\lesssim} h^{-2} \norm{\eeta_h^{i,n+1} - \eeta_h^{i,n}}{\LL{2}{\Omega}}.
\end{align}
Hence, the stopping criterion of \cite{Bartels2006,B06,Banas2008,Bavnas2008/09} implies the one used in Algorithm \ref{alg:effectivemidpoint}. Moreover, convergence $\norm{\eeta_h^{i,n+1} - \eeta_h^{i,n}}{\LL{2}{\Omega}} \rightarrow 0$ together with ~\eqref{eq:solverstep} proves that the repeat loop in step {\rm (ii)} of Algorithm \ref{alg:effectivemidpoint} terminates for some $n \in \N$. 

{\rm (v)} \label{item:firstsolverpih} For the Adams--Bashforth approach~\eqref{eq:appproach_ab} and the explicit Euler approach~\eqref{eq:approach_ee}, it holds that
\begin{align*}
\ppih{2\eeta_h^{i,n+1} - \mm_h^i}{\mm_h^i}{ \mm_h^{i-1}} - \ppih{2\eeta_h^{i,n} - \mm_h^i}{\mm_h^i}{ \mm_h^{i-1}} = \0,
\end{align*}
and~\eqref{eq:additionalassumption} is fulfilled (independently of $\boldsymbol{\pi}_{h}$ and $\boldsymbol{\pi}$).

{\rm (vi)} For uniaxial magnetocrystalline anisotropy and the stray field, we have proved in Section \ref{section:classicalcontributions} that we are in the situation of Theorem \ref{theorem:mainresult}~{\rm (c)}. Moreover the operator $\boldsymbol{\pi}_h:\LL{2}{\Omega} \rightarrow \LL{2}{\Omega}$ is linear and continuous in these cases. Hence,~\eqref{eq:additionalassumption} holds also for the implicit midpoint rule~\eqref{eq:approach_mp} even in the stronger form with $\norm{\eeta_h^{i,n+1} - \eeta_h^{i,n}}{\LL{2}{\Omega}}$ on the right-hand side.

{\rm (vii)} Finally, we show that~\eqref{eq:additionalassumption} is also satisfied for the Zhang--Li model from Section~\ref{subsection:zl} and the implicit midpoint approach~\eqref{eq:approach_mp}: According to {\rm (i)}, it holds that $\norm{\eeta_h^{i,n+1}}{\infty} \leq \norm{\mm_h^i}{\infty} = \norm{\mm_h^0}{\infty}$ for all $n \geq 0$. Together with the inverse inequality, we obtain
\begin{align*}
& \norm{\ppih{2\eeta_h^{i,n+1} - \mm_h^i}{\mm_h^i}{ \mm_h^{i-1}} - \ppih{2\eeta_h^{i,n} - \mm_h^i}{\mm_h^i}{ \mm_h^{i-1}} }{\LL{2}{\Omega}} \notag \\
& \quad \stackrel{\eqref{eq:defppizl}}{=} \norm{\eeta_h^{i,n+1}\times(\vv\cdot\nabla)\eeta_h^{i,n+1} + \xi(\vv\cdot\nabla)\eeta_h^{i,n+1} - \eeta_h^{i,n}\times(\vv\cdot\nabla)\eeta_h^{i,n} - \xi(\vv\cdot\nabla)\eeta_h^{i,n}}{\LL{2}{\Omega}}\\
& \quad \stackrel{\phantom{\eqref{eq:pihbounded}}}{\lesssim} 
\norm{ \mm_h^0 }{\LL{\infty}{\Omega}} \norm{ \nabla \eeta_h^{i,n+1} - \nabla \eeta_h^{i,n} }{\LL{2}{\Omega}} + 
h^{-1} \norm{ \mm_h^0 }{\LL{\infty}{\Omega}} \norm{ \eeta_h^{i,n+1} - \eeta_h^{i,n} }{\LL{2}{\Omega}} \notag \\
& \quad \stackrel{\phantom{\eqref{eq:pihbounded}}}{\lesssim} 
h^{-1} \norm{ \eeta_h^{i,n+1} - \eeta_h^{i,n} }{\LL{2}{\Omega}},
\end{align*}
i.e.,~\eqref{eq:additionalassumption} holds even in a stronger form.

\end{remark}

The following theorem extends Theorem \ref{theorem:mainresult} to Algorithm \ref{alg:effectivemidpoint}, where the error from the inexact solver is taken into account. The proof follows along the arguments of Section \ref{section:mainresult}.

\begin{theorem}\label{remark:mainremark}
Let the assumptions from Theorem \ref{theorem:mainresult}~{\rm (b)} and the additional assumptions from Remark \ref{remark:propalgeffective} {\rm (iv)} be fulfilled, i.e., it holds that $k = \mathbf{o}(h^2)$ as well as stability~\eqref{eq:additionalassumption}. Then, there hold the following assertions {\rm (a)}--{\rm (b)}.

{\rm (a)} \label{item:mainremark1} As $h,k,\epsilon \rightarrow 0$, there exists a subsequence of the postprocessed output $\mm_{hk}$ of Algorithm \ref{alg:effectivemidpoint} which converges weakly in $\HH{1}{\Omega_T}$ to some limit $\mm \in \HH{1}{\Omega_T}$ which is a weak solution to LLG in the sense of Definition \ref{def:weak}~\ref{item:weak1}--\ref{item:weak3}.

{\rm (b)} \label{item:mainremark2} In addition, suppose the assumptions of Theorem \ref{theorem:mainresult}~{\rm (c)}. Then, $\mm \in \HH{1}{\Omega_T}$ from~{\rm (a)} is a physical weak solution in the sense of Definition \ref{def:weak}~\ref{item:weak1}--\ref{item:weak5}.
\qed
\end{theorem}

\section{Numerical Experiments} \label{section:numerics}

This section provides some numerical experiments for Algorithm~\ref{alg:effectivemidpoint}. Our implementation is based on the C++/Python library Netgen/NGSolve~\cite{ngsolve}. To compute the stray field, we additionally build on the C++/Python library BEM++~\cite{Smigaj2015} in step~\ref{item:step2} of Algorithm~\ref{alg:fk}.
The visualization of the numerical results is done with ParaView \cite{Ahrens2005}.

\begin{figure}[t!]
\begin{subfigure}{0.48\textwidth}
\includegraphics[width=\linewidth]{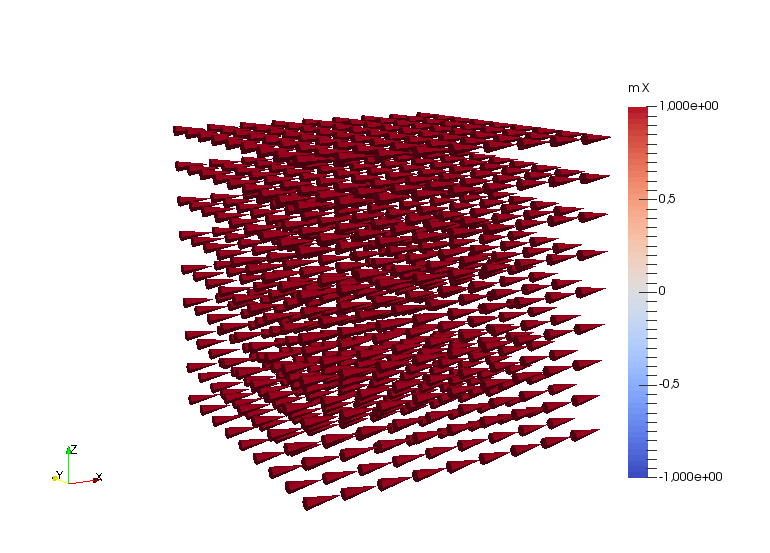}
\caption{$t=0$}
\label{subfig:emp1}
\end{subfigure}\hspace*{\fill}
\begin{subfigure}{0.48\textwidth}
\includegraphics[width=\linewidth]{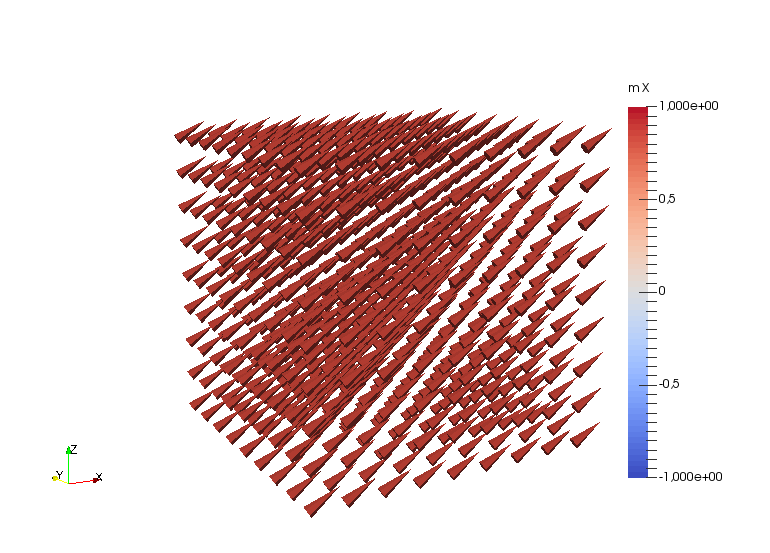}
\caption{$t=1$}
\label{subfig:emp2}
\end{subfigure}
\medskip
\begin{subfigure}{0.48\textwidth}
\includegraphics[width=\linewidth]{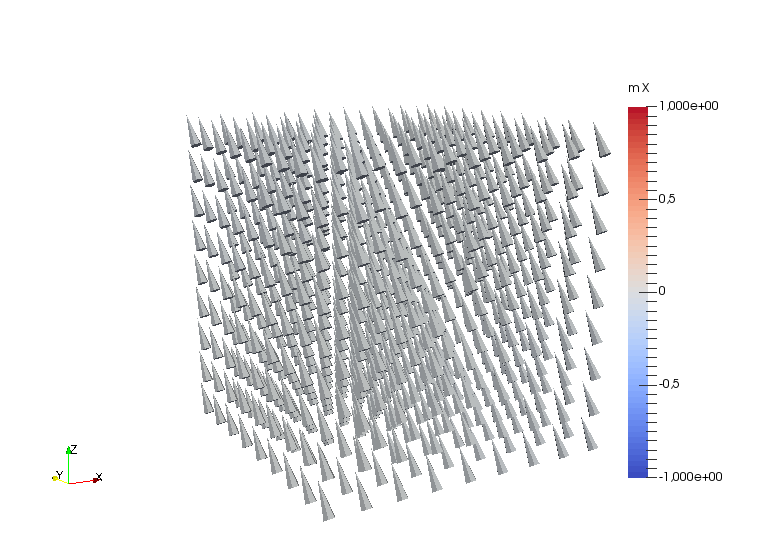}
\caption{$t=2$}
\label{subfig:emp3}
\end{subfigure}\hspace*{\fill}
\begin{subfigure}{0.48\textwidth}
\includegraphics[width=\linewidth]{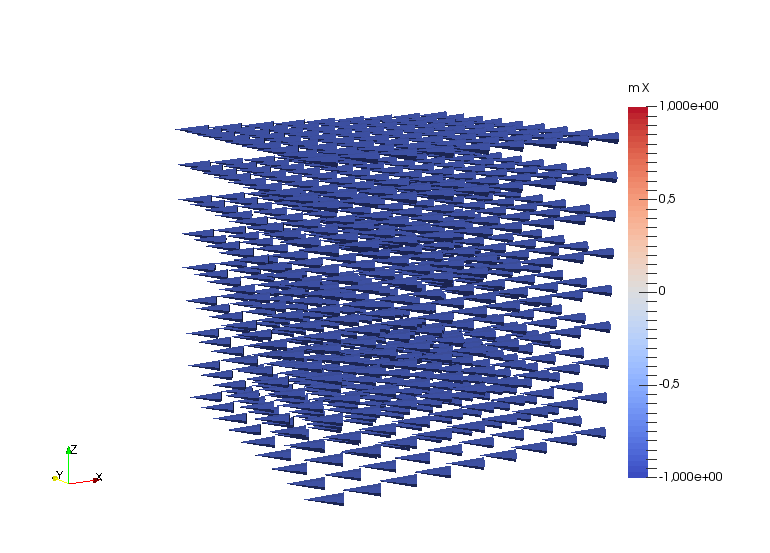}
\caption{$t=3$}
\label{subfig:emp4}
\end{subfigure}
\medskip
\begin{subfigure}{0.48\textwidth}
\includegraphics[width=\linewidth]{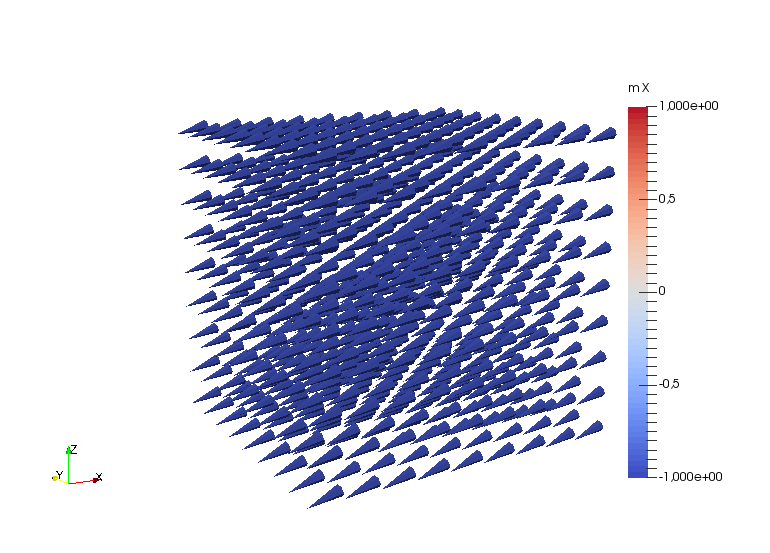}
\caption{$t=4$}
\label{subfig:emp5}
\end{subfigure}\hspace*{\fill}
\begin{subfigure}{0.48\textwidth}
\includegraphics[width=\linewidth]{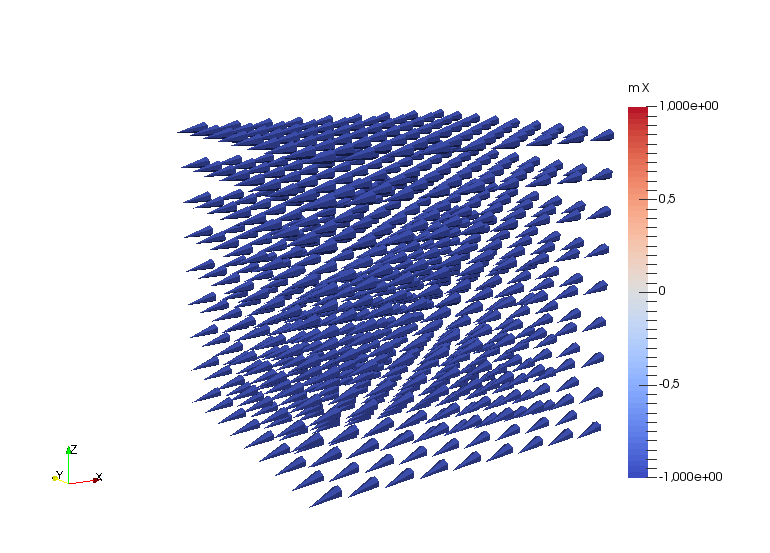}
\caption{$t=5$}
\label{subfig:emp6}
\end{subfigure}
\caption{Experiment of Section~\ref{section:experiment:academic}: Snapshots of the magnetization.}
\label{fig:whathappens}
\end{figure}

\begin{table}
 \begin{tabular}{p{0.15\textwidth}||p{0.15\textwidth}||p{0.15\textwidth}|p{0.15\textwidth}|p{0.15\textwidth}}
   & MP absolute & MP relative & AB relative & EE relative \\
  \hline\hline
  $k=0.0016$ 		&	$17.50$ &  $100\%$  & $100\%$  & $104.52\%$ \\
  \hline
  $k=0.0008$		&	 $8.27$ &  $100\%$  & $99.99\%$  & $104.85\%$  \\
  \hline
  $k=0.0004$		&	 $5.67$ &  $100\%$  & $100\%$  & $103.14\%$  \\
  \hline
  $k=0.0002$		&	 $4.34$ &  $100\%$  & $100\%$  & $107.00\%$  \\
  \hline
  $k=0.0001$		&	 $3.80$ &  $100\%$  & $100\%$  & $103.13\%$  \\
 \end{tabular}
 \caption{Experiment of Section~\ref{section:experiment:academic}: Number of iterations for one time-step and different treatments~\eqref{eq:different_approaches} of the stray field, where we provide the absolute numbers for the midpoint approach (MP absolute) as well as the relative numbers of the Midpoint approach (MP relative), the Adams--Bashforth approach (AB relative), and the explicit Euler approach (EE relative).}
 \label{table:average_iterations}
\end{table}

\begin{table}
 \begin{tabular}{p{0.15\textwidth}||p{0.15\textwidth}||p{0.15\textwidth}|p{0.15\textwidth}|p{0.15\textwidth}}
   & MP absolute & MP relative & AB relative & EE relative \\
   \hline\hline
   $0.0016$  & $1.21$  & $100\%$ & $60.00\%$ &  $62.22\%$ \\
   \hline
   $0.0008$  & $0.61$  & $100\%$ & $63.59\%$ &  $64.49\%$ \\
   \hline
   $0.0004$  & $0.47$  & $100\%$ & $60.40\%$ &  $63.33\%$ \\
   \hline
   $0.0002$  & $0.35$  & $100\%$ & $67.38\%$ &  $66.07\%$ \\
   \hline
   $0.0001$  & $0.31$  & $100\%$ & $64.16\%$ &  $64.72\%$ \\
 \end{tabular}
 \caption{Experiment of Section~\ref{section:experiment:academic}: Computational time for one time-step and different treatments~\eqref{eq:different_approaches} of the stray field, where we provide the absolute time (in s) for the Midpoint approach (MP absolute) as well as the relative times of the Midpoint approach (MP relative), the Adams--Bashforth approach (AB relative), and the explicit Euler approach (EE relative).}
 \label{table:average_duration}
\end{table}

\begin{figure}[ht]
\centering
\begin{tikzpicture}
\pgfplotstableread{plots/cumulative.dat}{\data}
\begin{axis}[
xlabel={Time},
ylabel={Duration in \si{\second}},
width = 100mm,
legend style={
legend pos= north west},
xmax=5,
xmin=0,
ymin=0,
ymax=4e3,
]
\addplot[red,ultra thick] table[x=time, y=ee] {\data};
\addplot[blue,ultra thick, dashed] table[x=time, y=ab] {\data};
\addplot[cyan,ultra thick] table[x=time, y=mp] {\data};
\legend{EE,AB,MP}
\end{axis}
\end{tikzpicture}
\caption{Experiment of Section~\ref{section:experiment:academic}: Cumulative computational time for $k=8\cdot10^{-4}$ and different treatments~\eqref{eq:different_approaches} of the stray field.}
\label{fig:duration}
\end{figure}
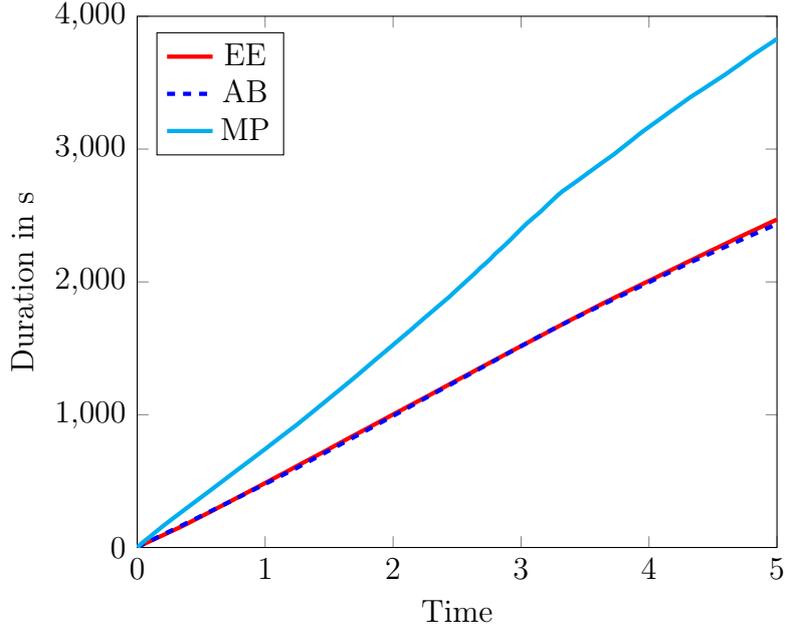
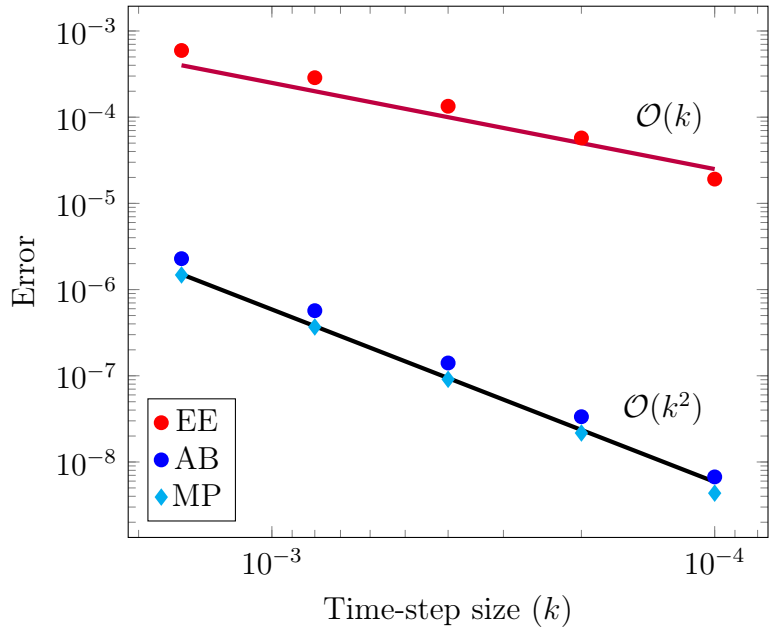
\begin{figure}[ht]
\centering
\begin{tikzpicture}
\pgfplotstableread{plots/rates.dat}{\data}
\begin{loglogaxis}[
xlabel={Time-step size ($k$)},
ylabel={Error},
width = 100mm,
legend style={
legend pos= south west},
x dir=reverse
]
\addplot[only marks,red,ultra thick] table[x=k, y=ee] {\data};
\addplot[only marks,blue,ultra thick] table[x=k, y=ab] {\data};
\addplot[only marks,cyan,mark=diamond*,ultra thick] table[x=k, y=mp] {\data};
\addplot[purple,ultra thick] table[x=k, y expr={(\thisrow{k})/4}]{\data};
\addplot[black,ultra thick] table[x=k, y expr={(\thisrow{k}*\thisrow{k})/1.7}]{\data};
\node at (axis cs:1e-4,2e-4) [anchor=north east] {$\mathcal{O}(k)$};
\node at (axis cs:1e-4,2e-8) [anchor=south east] {$\mathcal{O}(k^2)$};
\legend{EE,AB,MP} %
\end{loglogaxis}
\end{tikzpicture}
\caption{
Experiment of Section~\ref{section:experiment:academic}:
Reference error $\max_j \norm{\mm_{hk_{\textrm{ref}}}(t_j) - \mm_{hk}(t_j)}{\LL{2}{\Omega}}$ for different $k$ and treatments~\eqref{eq:different_approaches} of the stray field.
}
\label{fig:convergence_order}
\end{figure}

\subsection{Academic experiment and empirical convergence rates}
\label{section:experiment:academic}
This experiment aims to provide some insight into the accuracy and the computational effort for the different approaches~\eqref{eq:different_approaches} for the lower-order contributions. We consider LLG~\eqref{eq:LLG} in nondimensional form with $\Omega := (0,1)^3 \subset \R^3$, constant initial state $\mm^0 \equiv (1,0,0)$, constant external field $\ff \equiv (-2,-1/2,0)$, and finite time $T=5$. Besides exchange field and external field, the effective field $\heff$ also involves the stray field. We use a fixed uniform mesh $\Trian$ which consists of $3072$ tetrahedrons. Finally, we choose the parameter $\epsilon = 10^{-10}$ in Algorithm~\ref{alg:effectivemidpoint} to stop the iterative solver. 

Figure~\ref{fig:whathappens} shows some snapshots of the magnetization at times $t\in\{0,1,2,3,4,5\}$. As expected, the magnetization aligns with the applied externed field $\ff$ as time evolves.

We run Algorithm~\ref{alg:effectivemidpoint} for different time-step sizes $k := q \cdot 10^{-4}$ with $q \in \{1,2,4,8,16\}$. Table~\ref{table:average_iterations} provides the average number of fixed-point iterations per time-step of Algorithm~\ref{alg:effectivemidpoint}. As expected, the computational time decreases with the time-step size, since the fixed-point iteration in step~(iii) of Algorithm~\ref{alg:effectivemidpoint} then requires less steps until it terminates. We observe that the explicit treatment of the stray field by Adams--Bashforth~\eqref{eq:appproach_ab} resp. explicit Euler~\eqref{eq:approach_ee} roughly requires the same number of fixed-point iterations resp. increases the number of fixed-point iterations by about $5\%$ when compared to the implicit midpoint rule~\eqref{eq:approach_mp}. On the other hand, Table~\ref{table:average_duration} provides the computational times per time-step. Recall that the stray field computation by Algorithm~\ref{alg:fk} requires the solution of two additional linear systems plus the evaluation of a boundary integral operator. As expected, the treatment of the stray field by the implicit midpoint rule~\eqref{eq:approach_mp} is the most expensive approach, since the stray field is computed in each step of the fixed-point iteration. On the other hand, Adams--Bashforth~\eqref{eq:appproach_ab} and explicit Euler~\eqref{eq:approach_ee} lower the cost per time-step down to approximately $65\%$.  

Figure~\ref{fig:duration} displays the accumulation of the computational times until $T=5$ for $k=8\cdot10^{-4}$. Overall, the  explicit approaches by Adams--Bashforth~\eqref{eq:appproach_ab} and explicit Euler~\eqref{eq:approach_ee} only require $65\%$ of the computational time when compared to the implicit midpoint rule~\eqref{eq:approach_mp}.

Finally, Figure~\ref{fig:convergence_order} compares the different approaches~\eqref{eq:different_approaches} with respect to accuracy in terms of the experimental convergence rate.
Since the exact solution is unknown, we consider the error $\max_{0 \leq j \leq M} \norm{\mm_{hk_{\textrm{ref}}}(t_j) - \mm_{hk}(t_j)}{\LL{2}{\Omega}}$ with respect to a reference solution $\mm_{hk_{\textrm{ref}}}$, constructed from the output of Algorithm~\ref{alg:effectivemidpoint} for a finer time-step size $k_{\rm ref} := 5 \cdot 10^{-5}$.
As expected, both the implicit midpoint rule~\eqref{eq:approach_mp} and the explicit Adams--Bashforth~\eqref{eq:appproach_ab} approach exhibit second-order convergence and lead to approximately the same accuracy, while the treatment of the stray field by the explicit Euler approach~\eqref{eq:approach_ee} lowers the possible convergence rate down to linear.

Overall, the numerical results clearly underpin that the proposed explicit Adams--Bashforth approach~\eqref{eq:appproach_ab} is favorable for the treatment of the stray field.

\subsection{$\boldsymbol{\mu}$MAG standard problem \#5}
\label{section:mumag5}
To test our method for the simulation of practically relevant problem sizes, we consider the $\mu$MAG standard problem \#5, proposed by the Micromagnetic Modeling Activity Group~\cite{mumag} of the National Institute of Standards and Technology (NIST) of Gaithersburg (USA).

The computational domain is a ferromagnetic film $\widetilde{\Omega}$ with dimensions \SI[scientific-notation=false]{100}{\nano\meter} $\times$ \SI[scientific-notation=false]{100}{\nano\meter} $\times$ \SI[scientific-notation=false]{10}{\nano\meter}, aligned with the $x$, $y$, and $z$ axes of a Cartesian coordinate system, with origin at the center of the film.

The dynamics is driven by LLG with physical units and we make use of capital letters to distinguish it from the nondimensional form~\eqref{eq:LLG} of LLG:
\begin{subequations}\label{eq:LLG_physical}
\begin{align}
\partial_t \MM &= -\gamma_0 \MM \times \mathbf{H}_{\textrm{eff}} + \frac{\alpha}{M_s} \MM \times \partial_t \MM & &\textrm{ in } \left(0,T\right) \times \widetilde{\Omega}, \\
\partial_{\mathbf{n}} \MM &= \0 && \textrm{ on } \left(0,T\right) \times \partial \widetilde{\Omega}, \\
\MM(0) &= \MM^0 && \textrm{ in } \widetilde{\Omega}, 
\end{align}
where
\begin{align}
\label{eq:LLG_physical4}
\mathbf{H}_{\textrm{eff}} := \frac{2A}{\mu_0 M_s^2} \Delta\MM + \boldsymbol{\pi}(\MM)
\quad
\text{with }
\boldsymbol{\pi}(\MM) 
 := \mathbf{H}_{\textrm{s}}\vert_\Omega + \mm\times(\widetilde\vv\cdot\nabla)\mm + \xi\,(\widetilde\vv\times\nabla)\mm.
\end{align} 
\end{subequations}
The constants $\gamma_0 =$ \SI{2.21e5}{\meter\per\ampere\per\second} and $\mu_0 =$ \SI{4 \pi e-7}{\newton\per\ampere\squared} denote the gyromagnetic ratio and the magnetic permeability, respectively.
As for the material parameters, we consider the values of permalloy, i.e., $A=$ \SI{1.3e-11}{\joule\per\meter} for the exchange stiffness constant, $M_s=$ \SI{8.0e5}{\ampere\per\meter} for the saturation magnetization, and $\alpha=$~\num[scientific-notation=false]{0.1} for the damping parameter.
The lower-order terms in~\eqref{eq:LLG_physical4} comprise the stray field $\mathbf{H}_{\textrm{s}} = - \nabla u$, where the magnetostatic potential $u$ is the solution of the transmission problem~\eqref{eq:magnetostatic} for $\mm=\MM/M_s$, as well as the Zhang--Li contribution, with $\widetilde{\vv}$ being the spin velocity vector (in~\si{\meter\per\second}) and $\xi>0$ the ratio of nonadiabacity.
The initial state is obtained by solving~\eqref{eq:LLG_physical} for the initial condition $\widetilde\MM_0(x,y,z) = M_s\,(-y,x,10) / \sqrt{x^2 + y^2 + 100}$ and $\widetilde{\vv} = \0$ for a sufficiently long time, until the equilibrium configuration is reached; see Figure~\ref{subfig:mumag5_initial}.
Then, for $t \geq 0$, we set $\widetilde{\vv} := (-72.17,0,0)$ and $\xi := 0.05$, and simulate the system until the system reaches the new equilibrium (the choice $T=$ \SI{8}{\nano\second} is sufficient); see Figure~\ref{subfig:mumag5_end}.

With the scaling parameter $L:=10^{-9}$, the function $\mm := \MM / M_s$ fulfills the nondimensional Gilbert form~\eqref{eq:LLG} with 
\begin{align}
\Cex := \frac{2A}{\mu_0 M_s^2 L^2} \quad \textrm{and} \quad
\vv := - \frac{1}{\gamma_0 M_s L} \widetilde{\vv}
\end{align}
in~\eqref{eq:defppizl} and $\Omega := (-50,50)^2\times(-5,-5)$.

To discretize the problem, we employ a regular triangulation of $\Omega$ into approximately $25 000$ elements and choose $k$ in order to obtain a time-step size of \SI{0.005}{\pico\second} in physical units.
We use $\epsilon = 5 \cdot 10^{-5}$ in Algorithm~\ref{alg:effectivemidpoint} to stop the iterative solver.
For the lower-order contributions, we use the explicit Adams--Bashforth approach~\eqref{eq:appproach_ab}.

In Figure~\ref{fig:average_xy}, we plot the time evolution of the averaged value of the $x$- resp.\ $y$-component of $\mm$, and compare our results with those obtained with OOMMF~\cite{dp1999}.
Due to the different nature of the considered methods (e.g., FEM vs.\ FDM for the spatial discretization, FFT vs.\ Algorithm~\ref{alg:fk} for the computation of the stray field, adaptive vs.\ uniform time-stepping), we cannot expect a perfect quantitative agreement of the simulation results.
However, the comparison shows that the qualitative behavior of the solutions is preserved.

\begin{figure}[t!]

\begin{subfigure}{0.48\textwidth}
\includegraphics[width=0.9\textwidth]{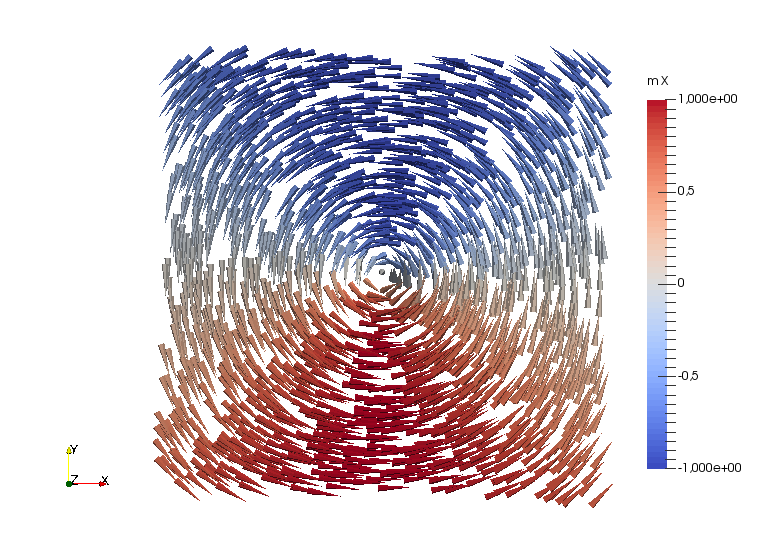}
\caption{Initial vortex state at $t =$ \SI{0}{\nano\second}.}
\label{subfig:mumag5_initial}
\end{subfigure}\hspace*{\fill}
\begin{subfigure}{0.48\textwidth}
\includegraphics[width=0.9\textwidth]{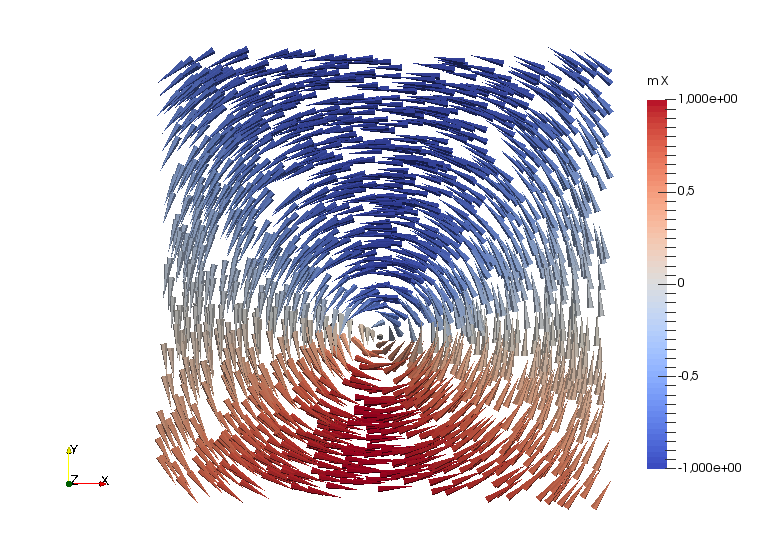}
\caption{Final vortex state at $t =$ \SI{8}{\nano\second}.}
\label{subfig:mumag5_end}
\end{subfigure}
\caption{$\mu$MAG standard problem \#5: Magnetization at different time-steps.}
\end{figure}

\begin{figure}[ht]
\centering
\begin{tikzpicture}
\pgfplotstableread{plots/sp5-oommf.dat}{\oommf}
\pgfplotstableread{plots/sp5-mp.dat}{\mp}
\begin{axis}[
width = 100mm,
xlabel={Time (in \si{\nano\second})},
xmin=0,
xmax=8,
ymin=-0.4,
ymax=0.4,
]
\addplot[blue,ultra thick] table[x=t, y=mx]{\mp};
\addplot[cyan,ultra thick] table[x=t, y=my]{\mp};
\addplot[purple,dashed,ultra thick] table[x=t, y=mx]{\oommf};
\addplot[red,dashed,ultra thick] table[x=t, y=my]{\oommf};
\legend{NGS $\langle \mm_x \rangle$, NGS $\langle \mm_y \rangle$, OOMMF $\langle \mm_x \rangle$, OOMMF $\langle \mm_y \rangle$}
\end{axis}
\end{tikzpicture}
\caption{$\mu$MAG standard problem \#5: Comparison of the results obtained by Algorithm~\ref{alg:effectivemidpoint} with OOMMF.}
\label{fig:average_xy}
\end{figure}
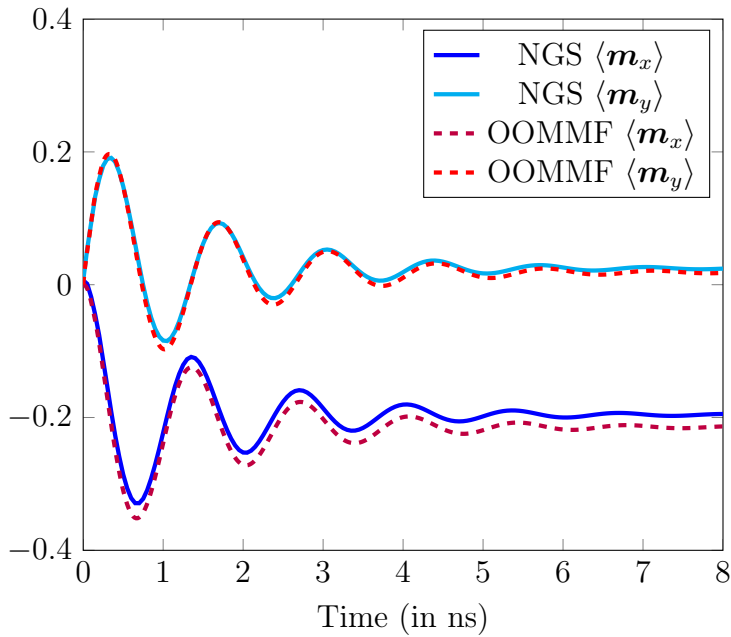

\section*{}

{\bf Acknowledgement.} The authors acknowledge support of the Vienna Science and Technology fund (WWTF) under grant MA14-44, of the Austrian Science Fund (FWF) under grant W1245, and of TU Wien through the innovative projects initiative.
We thank Alexander Rieder (TU Wien) and Alexander Haberl (TU Wien) for their help with coupling NGSolve to the BEM++ library.

\bibliographystyle{alpha}
\bibliography{ref}
\end{document}

%% file: macros.tex
\def\abs#1{|#1|}
\def\norm#1#2{\|#1\|_{#2}}

\def\prod#1#2{{\langle #1,#2\rangle}}
\def\prodh#1#2{{\langle #1,#2\rangle}_h}
\def\heff{\boldsymbol{h}_{\operatorname{eff}}}
\def\Deltah#1{\Delta_h #1}
\def\diam#1{\operatorname{diam} (#1)}

\def\diver{\operatorname{div}}
\def\Ppih#1{\mathbb{P}_h #1}
\def\ppi#1{\boldsymbol{\pi}( #1 )}

\def\ppih#1#2#3{\boldsymbol{\Pi}_h (#1,#2,#3)}
\def\ppihmi{\boldsymbol{\Pi}_{h}^i}
\def\ppihmj{\boldsymbol{\Pi}_{h}^j}
\def\ppihkm{\boldsymbol{\Pi}_{hk}}

\def\mid#1#2{#1^{#2 + \frac12}}
\def\midh#1#2{#1_h^{#2 + \frac12}}
\def\halfh#1#2{#1_h^{#2 + \frac12}}

\def\dt#1#2{{\mathop{\mathrm{d_t}#1}}^{#2+1}}
\def\dth#1#2{ \mathop{\mathrm{d_t}} { #1_h^{#2+1} } }
\def\dtshort{\mathop{\mathrm{d_t}}}

\def\R{\mathbb{R}}

\def\0{\boldsymbol{0}}

\def\NN{\mathcal N}

\def\H{\boldsymbol{H}}

\def\Cex{C_{\mathrm{\bf ex}}}
\def\Cmesh{C_{\mathrm{\bf mesh}}}

\def\hh{\boldsymbol{h}}

\def\mm{\boldsymbol{m}}
\def\nn{\boldsymbol{n}}

\def\vv{\boldsymbol{v}}

\def\HH{\boldsymbol{H}}
\def\ff{\boldsymbol{f}}
\def\vv{\boldsymbol{v}}

\def\xx{\boldsymbol{x}}
\def\yy{\boldsymbol{y}}
\def\zz{\boldsymbol{z}}
\def\FF{\boldsymbol{F}}

\def\cchi{\boldsymbol{\chi}}
\def\vvarphi{\boldsymbol{\varphi}}
\def\ppsi{\boldsymbol{\psi}}

\def\Ppi{\boldsymbol{\mathbb{P}}}
\def\Ttetha{\boldsymbol{\Theta}}
\def\eeta{\boldsymbol{\eta}}
\def\II{\boldsymbol{\mathcal{I}}}
\def\MM{\boldsymbol{M}}
\def\VV{\boldsymbol{V}}

\def\EE#1#2{\mathcal{E} (#1,#2)}

\def\L#1#2{L^#1 (#2) }
\def\LL#1#2{\boldsymbol{L}^#1 (#2) }
\def\H#1#2{H^#1 (#2) }
\def\HH#1#2{\boldsymbol{H}^#1 (#2) }

\def\Vh{\boldsymbol{V}_{\!\!h}}
\def\Sh{S_h}

\def\Ltwoshort{\boldsymbol{L}^2}

\def\ConeL{C^1 \left( \left[0,T\right], \boldsymbol{L}^2 \left(\Omega\right) \right)}

\def\Int#1#2{\int_{#1}^{#2}}
\def\IntSet#1{\int_{#1}}
\def\Sum#1#2{\sum_{#1}^{#2}}
\def\SumSet#1{\sum_{#1}}
\def\d#1{\mathop{\mathrm{d}#1}} 

\def\Nh{\mathcal{N}_h}
\def\N{\mathbb{N}}
\def\R{\mathbb{R}}

\def\lr#1{(#1)}

\def\ltr{\!\left(t\right)}

\def\l0r{\!\left(0\right)}
\def\Trian{\mathcal{T}_h}
\def\ashkzero{\textrm{ as } (h,k) \rightarrow (0,0)}
\def\textashkzero{as $(h,k) \rightarrow (0,0)$}
\def\limhk{\lim\limits_{(h,k) \rightarrow (0,0)}}
\def\liminfhk{\liminf\limits_{(h,k) \rightarrow (0,0)}}

%% file: midpoint.bbl
\newcommand{\etalchar}[1]{$^{#1}$}
\def\soft#1{\leavevmode\setbox0=\hbox{h}\dimen7=\ht0\advance \dimen7
  by-1ex\relax\if t#1\relax\rlap{\raise.6\dimen7
  \hbox{\kern.3ex\char'47}}#1\relax\else\if T#1\relax
  \rlap{\raise.5\dimen7\hbox{\kern1.3ex\char'47}}#1\relax \else\if
  d#1\relax\rlap{\raise.5\dimen7\hbox{\kern.9ex \char'47}}#1\relax\else\if
  D#1\relax\rlap{\raise.5\dimen7 \hbox{\kern1.4ex\char'47}}#1\relax\else\if
  l#1\relax \rlap{\raise.5\dimen7\hbox{\kern.4ex\char'47}}#1\relax \else\if
  L#1\relax\rlap{\raise.5\dimen7\hbox{\kern.7ex
  \char'47}}#1\relax\else\message{accent \string\soft \space #1 not
  defined!}#1\relax\fi\fi\fi\fi\fi\fi}
\begin{thebibliography}{AHP{\etalchar{+}}14}

\bibitem[AGL05]{Ahrens2005}
J.~Ahrens, B.~Geveci, and C.~Law.
\newblock {\em ParaView: An End-User Tool for Large Data Visualization.}
\newblock Visualization Handbook, Elsevier, 2005.

\bibitem[AHP{\etalchar{+}}14]{Abert2014}
C.~Abert, G.~Hrkac, M.~Page, D.~Praetorius, M.~Ruggeri, and D.~Suess.
\newblock Spin-polarized transport in ferromagnetic multilayers: an
  unconditionally convergent {FEM} integrator.
\newblock {\em Comput. Math. Appl.}, 68(6):639--654, 2014.

\bibitem[AKT12]{AKT2012}
F.~Alouges, E.~Kritsikis, and J.-C. Toussaint.
\newblock A convergent finite element approximation for
  {L}andau-{L}ifschitz-{G}ilbert equation.
\newblock {\em Physica B}, 407:1345--1349, 2012.

\bibitem[Alo08]{Alouges2008}
F.~Alouges.
\newblock A new finite element scheme for {L}andau-{L}ifchitz equations.
\newblock {\em Discrete Contin. Dyn. Syst. Ser. S}, 1(2):187--196, 2008.

\bibitem[AS92]{Alouges1992}
F.~Alouges and A.~Soyeur.
\newblock On global weak solutions for {L}andau-{L}ifshitz equations: existence
  and nonuniqueness.
\newblock {\em Nonlinear Anal.}, 18(11):1071--1084, 1992.

\bibitem[Bar06]{B06}
S.~Bartels.
\newblock Constraint preserving, inexact solution of implicit discretizations
  of {L}andau–{L}ifshitz–{G}ilbert equations and consequences for
  convergence.
\newblock {\em PAMM}, 6(1):19--22, 2006.

\bibitem[Bar15]{Bartels2015}
S.~Bartels.
\newblock {\em Numerical methods for nonlinear partial differential equations},
  volume~47 of {\em Springer Series in Computational Mathematics}.
\newblock Springer, Cham, 2015.

\bibitem[BBP08]{Banas2008}
{\soft{L}}.~Ba{\v{n}}as, S.~Bartels, and A.~Prohl.
\newblock A convergent implicit finite element discretization of the
  {M}axwell-{L}andau-{L}ifshitz-{G}ilbert equation.
\newblock {\em SIAM J. Numer. Anal.}, 46(3):1399--1422, 2008.

\bibitem[BP06]{Bartels2006}
S.~Bartels and A.~Prohl.
\newblock Convergence of an implicit finite element method for the
  {L}andau-{L}ifshitz-{G}ilbert equation.
\newblock {\em SIAM J. Numer. Anal.}, 44(4):1405--1419 (electronic), 2006.

\bibitem[BPP15]{bpp}
{\soft{L}}.~Ba{\v{n}}as, M.~Page, and D.~Praetorius.
\newblock A convergent linear finite element scheme for the
  {M}axwell-{L}andau-{L}ifshitz-{G}ilbert equations.
\newblock {\em Electron. Trans. Numer. Anal.}, 44:250--270, 2015.

\bibitem[BPPR14]{bppr}
{\soft{L}}.~Ba{\v{n}}as, M.~Page, D.~Praetorius, and J.~Rochat.
\newblock A decoupled and unconditionally convergent linear {FEM} integrator
  for the {L}andau-{L}ifshitz-{G}ilbert equation with magnetostriction.
\newblock {\em IMA J. Numer. Anal.}, 34(4):1361--1385, 2014.

\bibitem[BPS09]{Bavnas2008/09}
{\soft{L}}.~Ba{\v{n}}as, A.~Prohl, and M.~Slodi{\v{c}}ka.
\newblock Modeling of thermally assisted magnetodynamics.
\newblock {\em SIAM J. Numer. Anal.}, 47(1):551--574, 2008/09.

\bibitem[BPS12]{Bavnas2012}
{\soft{L}}.~Ba{\v{n}}as, A.~Prohl, and M.~Slodi{\v{c}}ka.
\newblock Numerical scheme for augmented {L}andau-{L}ifshitz equation in heat
  assisted recording.
\newblock {\em J. Comput. Appl. Math.}, 236(18):4775--4787, 2012.

\bibitem[BSF{\etalchar{+}}14]{Bruckner2014}
F.~Bruckner, D.~Suess, M.~Feischl, T.~F{\"u}hrer, P.~Goldenits, M.~Page,
  D.~Praetorius, and M.~Ruggeri.
\newblock Multiscale modeling in micromagnetics: existence of solutions and
  numerical integration.
\newblock {\em Math. Models Methods Appl. Sci.}, 24(13):2627--2662, 2014.

\bibitem[CEF11]{Carbou2011}
G.~Carbou, M.~Efendiev, and P.~Fabrie.
\newblock Global weak solutions for the {L}andau-{L}ifschitz equation with
  magnetostriction.
\newblock {\em Math. Methods Appl. Sci.}, 34(10):1274--1288, 2011.

\bibitem[CF98]{Carbou1998}
G.~Carbou and P.~Fabrie.
\newblock Time average in micromagnetism.
\newblock {\em J. Differential Equations}, 147(2):383--409, 1998.

\bibitem[CF01]{cf2001}
G.~Carbou and P.~Fabrie.
\newblock Regular solutions for {L}andau-{L}ifschitz equation in a bounded
  domain.
\newblock {\em Differential Integral Equations}, 14(2):213--229, 2001.

\bibitem[DP99]{dp1999}
M.~J. Donahue and D.~G. Porter.
\newblock {OOMMF} user's guide, {V}ersion 1.0.
\newblock Interagency Report NISTIR 6376, National Institute of Standards and
  Technology, Gaithersburg, MD, 1999.

\bibitem[DS14]{DS2014}
E.~Dumas and F.~Sueur.
\newblock On the weak solutions to the {M}axwell-{L}andau-{L}ifshitz equations
  and to the {H}all-{M}agneto-{H}ydrodynamic equations.
\newblock {\em Commun. Math. Phys.}, 330:1179--1225, 2014.

\bibitem[Eva10]{Evans2010}
L.~C. Evans.
\newblock {\em Partial differential equations}, volume~19 of {\em Graduate
  Studies in Mathematics}.
\newblock American Mathematical Society, Providence, RI, second edition, 2010.

\bibitem[FK90]{Fredkin1990}
D.~Fredkin and T.~Koehler.
\newblock Hybrid method for computing demagnetizing fields.
\newblock {\em IEEE Trans. Magn.}, 26(2):415--417, 1990.

\bibitem[FT17]{FT17}
M.~Feischl and T.~Tran.
\newblock The eddy current-{LLG} equations: {FEM}-{BEM} coupling and a priori
  error estimates.
\newblock {\em SIAM J. Numer. Anal.}, 55(4):1786--1819, 2017.

\bibitem[GCW07]{Garcia-Cervera2007}
C.~J. Garc{\'{\i}}a-Cervera and X.-P. Wang.
\newblock Spin-polarized transport: existence of weak solutions.
\newblock {\em Discrete Contin. Dyn. Syst. Ser. B}, 7(1):87--100, 2007.

\bibitem[Gol12]{Goldenits2012}
P.~Goldenits.
\newblock {\em Konvergente numerische Integration der Landau-Lifshitz-Gilbert
  Gleichung}.
\newblock PhD thesis, TU Wien, Institute for Analysis and Scientific Computing,
  2012.

\bibitem[HS98]{Hubert1998}
A.~Hubert and R.~Schäfer.
\newblock {\em Magnetic domains; the analysis of magnetic microstructures}.
\newblock Springer, Berlin, 1998.
\newblock Corrected Printing 2000.

\bibitem[LPPT15]{lppt}
K.-N. Le, M.~Page, D.~Praetorius, and T.~Tran.
\newblock On a decoupled linear {FEM} integrator for eddy-current-{LLG}.
\newblock {\em Appl. Anal.}, 94(5):1051--1067, 2015.

\bibitem[LT13]{ellg}
K.-N. Le and T.~Tran.
\newblock A convergent finite element approximation for the quasi-static
  {M}axwell-{L}andau-{L}ifshitz-{G}ilbert equations.
\newblock {\em Comput. Math. Appl.}, 66(8):1389--1402, 2013.

\bibitem[MP13]{Melcher2013}
C.~Melcher and M.~Ptashnyk.
\newblock Landau-{L}ifshitz-{S}lonczewski equations: global weak and classical
  solutions.
\newblock {\em SIAM J. Math. Anal.}, 45(1):407--429, 2013.

\bibitem[mum]{mumag}
$\mu${MAG} -- {M}icromagnetic {M}odeling {A}ctivity {G}roup. {N}ational
  {I}nstitute for {S}tandards and {T}echnology ({N}{I}{S}{T}).
\newblock \url{http://www.ctcms.nist.gov/~rdm/mumag.org.html}.
\newblock Accessed: 2016-10-25.

\bibitem[Pra04]{Pra04}
D.~Praetorius.
\newblock Analysis of the operator {$\Delta^{-1}{\rm div}$} arising in magnetic
  models.
\newblock {\em Z. Anal. Anwendungen}, 23(3):589--605, 2004.

\bibitem[QV94]{Quarteroni1994}
A.~Quarteroni and A.~Valli.
\newblock {\em Numerical approximation of partial differential equations},
  volume~23 of {\em Springer Series in Computational Mathematics}.
\newblock Springer-Verlag, Berlin, 1994.

\bibitem[{\'S}BA{\etalchar{+}}15]{Smigaj2015}
W.~{\'S}migaj, T.~Betcke, S.~Arridge, J.~Phillips, and M.~Schweiger.
\newblock Solving boundary integral problems with {BEM}++.
\newblock {\em ACM Trans. Math. Software}, 41(2):Art. 6, 40, 2015.

\bibitem[Sch]{ngsolve}
J.~Schöberl.
\newblock {NGS}olve finite element library.
\newblock \url{https://ngsolve.org/}.
\newblock Accessed: 2017-11-04.

\bibitem[SS11]{ss2011}
S.~A. Sauter and C.~Schwab.
\newblock {\em Boundary element methods}, volume~39 of {\em Springer Series in
  Computational Mathematics}.
\newblock Springer-Verlag, Berlin, 2011.
\newblock Translated and expanded from the 2004 German original.

\bibitem[SZ90]{Scott1990}
L.~R. Scott and S.~Zhang.
\newblock Finite element interpolation of nonsmooth functions satisfying
  boundary conditions.
\newblock {\em Math. Comp.}, 54(190):483--493, 1990.

\bibitem[TNMS05]{tnms2005}
A.~Thiaville, Y.~Nakatani, J.~Miltat, and Y.~Suzuki.
\newblock Micromagnetic understanding of current-driven domain wall motion in
  patterned nanowires.
\newblock {\em EPL (Europhysics Letters)}, 69(6):990, 2005.

\bibitem[Vis85]{visintin}
A.~Visintin.
\newblock On {L}andau-{L}ifshitz' equations for ferromagnetism.
\newblock {\em Japan J. Appl. Math.}, 2(1):69--84, 1985.

\bibitem[ZL04]{zl2004}
S.~Zhang and Z.~Li.
\newblock Roles of nonequilibrium conduction electrons on the magnetization
  dynamics of ferromagnets.
\newblock {\em Phys. Rev. Lett.}, 93(12):127204, 2004.

\end{thebibliography}
